\documentclass[12pt,fleqn]{iopart}

\usepackage{rotating}
\usepackage{grf}
\usepackage{tikz}
\usepackage[utf8]{inputenc}
\usepackage[T1]{fontenc}
\usepackage{iopams}
\expandafter\let\csname equation*\endcsname\relax
\expandafter\let\csname endequation*\endcsname\relax
\usepackage{amsmath}
\usepackage{amsfonts}
\usepackage{amsopn}
\usepackage{amsthm}

\usepackage{doi}
\usepackage{url}
\usepackage[numbers,sort,compress]{natbib}

\usepackage{hyperref}
\hypersetup{
   unicode=true,%         non-Latin characters in Acrobat’s bookmarks
   pdftoolbar=false,%     show Acrobat’s toolbar?
   pdfmenubar=false,%     show Acrobat’s menu?
   pdffitwindow=false,%   page fit to window when opened
   pdfnewwindow=true,%    links in new window
   pdfkeywords={},%       list of keywords
   colorlinks=true,%      false: boxed links; true: colored links
   linkcolor=black,%      color of internal links
   citecolor=black,%      color of links to bibliography
   filecolor=black,%      color of file links
   urlcolor=black%        color of external links
}
\usepackage{xspace}
\usepackage{soul}

\newcommand{\vect}{\boldsymbol}

\newcommand{\expect}{\mathbb E}
\newcommand{\prob}{\mathbb P}
\newcommand{\elias}[2]{{\color{black}{#1}}}

\def\eye{\mathbb I}
\def\eqref#1{{\rm (\ref{#1})}}

\DeclareMathOperator{\diag}{diag}
\DeclareMathOperator{\subjto}{s.t.:}

\DeclareMathOperator*{\argmin}{arg\ min}

\newtheorem{proposition}{Proposition}

\newtheorem{lemma}{Lemma}

\begin{document}

   \title[Bregman-Risk Estimators: Application to Parameter Selection]{Unbiased Bregman-Risk Estimators: Application to Regularization Parameter Selection in Tomographic Image Reconstruction}

   \author{Elias S. Helou, Sandra A. Santos, and Lucas E. A. Simões}
   \address{\textsc{sme/icmc/usp}, Postal Box 668, 13560-970, São Carlos, SP, Brazil}
   \ead{elias@icmc.usp.br}

   \begin{abstract}
      Unbiased estimators are introduced for averaged Bregman divergences which generalize Stein's Unbiased (Predictive) Risk Estimator, and the minimization of these estimators is proposed as a regularization parameter selection method for regularization of inverse problems. Numerical experiments are presented in order to show the performance of the proposed technique. Experimental results indicate a useful occurence of a concentration of measure phenomena and some implications of this hypothesis are analyzed.
   \end{abstract}
   \vspace{2pc}

%    \maketitle

   \submitto{Inverse Problems}

   \section{Introduction}

   Many problems in science and engineering can be formulated as a system of nonlinear equations of the form
   \begin{equation}\label{eq:system}
      \vect A( \vect x ) \approx \vect b,
   \end{equation}
   where $\vect x \in \mathbb R^n$ is the vector of unknowns, $\vect A : \mathbb R^n \mapsto \mathbb R^m$ is the system function arising from a mathematical model for the problem, and $\vect b \in \mathbb R^m$ is the vector of observed data, which contains noise, that is, it is given by
   \begin{equation}\label{eq:model}
      \vect b = \vect A( \vect x^* ) + \vect\epsilon,
   \end{equation}
   where $\vect x^* \in \mathbb R^n$ is the exact solution and $\vect\epsilon$ is some unknown vector of random variables. The methodology we will propose can be applied to several noise models, including Poisson distributed, the sum of Gauss distributed and Poisson distributed, exponential family distributed and elliptically contoured distributed data. Example applications include tomographic image reconstruction~\cite{her80,nat86,kas88} and image denoising and deblurring~\cite{naw01,bov05}.

   \elias{}{When solving a system of linear equations, which is usually a simple task, being feasible using regular hardware and software for systems with several millions of variables, difficulties arise when the system matrix is ill-conditioned and the right-hand-side data $\vect b$ is contaminated by noise. The aforementioned tasks of tomographic image reconstruction and image denoising and deblurring are, to different degrees\footnote{Limited-angle tomography is extremely ill-conditioned, while denoising is not ill-conditioned at all.}, ill-conditioned problems for which the data often contains a noticeable amount of noise. In this case, the error contained in $\vect b$ may be severely amplified in the solution of system~\eqref{eq:system}, an effect that could reduce by much the usefulness of the resulting images.}

   Because ill-conditioned inverse problems arise so often in applications, methods for obtaining meaningful results from noisy data have been devised. These techniques are the so-called \emph{regularization methods}~\cite{ehn00} and always require a \emph{regularization parameter} to be selected by the user. In the present paper, we develop techniques for estimating certain expected errors and we apply these techniques to the problem of selecting parameters for nonlinear approaches to regularization.

   Among the first regularization methods, we have Tikhonov regularization~\cite{ehn00}, which consists \elias{of}{in} using, as an approximation to the unknown solution, the minimizer $\vect x_{\text{\tiny Tikhonov}}^\gamma$ of
   \begin{equation}\label{eq:tykhnov_classic}
         \min \| A\vect x - \vect b \|_2^2 + \gamma \| \vect x \|_2^2.
   \end{equation}
   Here, $\gamma \geq 0$ is the regularization parameter, which determines how much smoothness is imposed \elias{on}{to} the regularized solution, and we are assuming that the problem is linear, i.e., $A$ is a matrix. It is possible to verify that the Tikhonov regularized solution is given by
   \begin{equation}
      \vect x_{\text{\tiny Tikhonov}}^\gamma = (A^T A  + \gamma \eye )^{-1}A^T\vect b,
   \end{equation}
   where we write $\eye$ for the identity matrix, which will always have its dimension clear from the context. Unlike Tikhonov's, many regularization techniques are nonlinear. For example, maximum likelihood and penalized maximum likelihood methods are \elias{}{becoming more} common for tomographic image reconstruction from emission data~\cite{ahf03,hed05,hcc14,hul94,dey01,brd96}. In this case, the regularized approximation is the solution of a more general convex optimization problem, possibly of the form
   \begin{eqnarray}
      \min & \| M \vect x \|_1\nonumber \\
      \subjto  & - \log \prob( \vect b | \vect x ) \leq \ell,\quad \vect x \in \mathbb R_+^n\label{eq:minell1}
   \end{eqnarray}
   where $M$ is some sparsifying decomposition (such as a wavelet transform) and $\prob( \cdot | \vect x )$ is the probability density function of the data for a given image $\vect x$. The parameter now is $\ell > 0$, which determines our ``tolerance to unlikelihood'', within which the (hopefully) sparsest solution in the transformed space will be selected. We will study the general situation where the regularized solution is given by
   \begin{equation}\label{eq:reg}
      \vect x^\gamma := \vect B_\gamma( \vect b ),
   \end{equation}
   where $\vect B_\gamma : \mathbb R^m \mapsto \mathbb R^n$ is a function implicitly defined by the regularization method. Precise conditions on $\vect B_\gamma$ for our techniques to be applicable depend on the noise model, and will be discussed later in the text, when appropriate.

   Because confidence in the regularized solution $\vect x^\gamma$ requires careful choice of $\gamma$, several techniques to accomplish this task have been developed and studied for linear regularization, under the assumptions of uncorrelated Gaussian or deterministic noise~\cite{deg95,ghw79,han92,hat87,tbk91,bal11,bej20}. Data from imaging technologies, however, often do not fit well under the pure Gaussian hypothesis, because signal detection for imaging techniques is usually performed as a photon-counting process, thereby leading to data that can be modeled as Poisson variables~\cite{vsk85}, possibly with added Gaussian noise from underlying signal amplifying electronics~\cite{snh93}. Methods for regularization parameter selection under Poissonian, mixed Poissonian-Gaussian and more general noise hypotheses have also been devised, including those designed for nonlinear regularization~\cite{bag09,bbt10,lbu11,eld09}.

   Many parameter selection techniques rely on Unbiased Predictive Risk Estimators (UPREs) or Unbiased Risk Estimators (UREs), that is, computable unbiased estimators for
   \begin{equation}\label{eq:UPRES}
      \expect \| \vect A( \vect x^\gamma ) - \vect A( \vect x^* ) \|_2^2 \quad \text{or} \quad \expect \| \vect x^\gamma - \vect x^* \|_2^2,
   \end{equation}
   where $\expect := \expect_{\vect b}$ is the probabilistic expectation over the random variable $\vect b$, on which $\vect x^\gamma$ depends. Such U(P)REs are minimized with \elias{respect}{relation} to the regularization parameter to yield a selection rule~\cite{bbp05,deg95,gak92}. \elias{When $\vect A$ is linear, the regularization is obtained via an \emph{ordered filter} (encompassing a large class of linear regularization operators), and the noise is Gaussian, it can be shown that selecting the minimizer of certain unbiased estimators, such as the UPRE regularization parameter, relates to the URE~\cite{liw2020} as the magnitude of the error diminishes. It is unclear what is the most general setting where this kind of strong relationship between UPREs and UREs will hold.}{}
   
   Notice that the averaged quantities in~\eqref{eq:UPRES} are not directly computable and Stein's Lemma~\cite{jas61,ste81} is the tool which allows obtaining U(P)REs, originally for the i.i.d.\ Gaussian case, but later generalized for a variety of noise models~\cite{eld09,lbu11,mao14,avh06,hav08,lan08}. We will focus on predictive-type estimators, but it is possible to obtain projected UREs from these operators, at least in the linear model case~\cite{eld09}.

   \paragraph{Contributions of the paper} The main contribution of the present paper is the introduction of attainable unbiased estimators for general risk functions, which include: %não tem espaço extra!
   \begin{equation}\label{eq:UBPRES}
      \expect D_f\bigl( \vect A( \vect x^* ), \vect A( \vect x^\gamma ) \bigr),
   \end{equation}
   for $f : \mathbb R^n \mapsto \mathbb R$, where $D_f$ is a Bregman divergence~\cite{bre67}:
   \begin{equation}\label{eq:bregdiv}
      D_f( \vect x, \vect y ) = f( \vect x ) - f( \vect y ) - \nabla f( \vect y )^T( \vect x - \vect y ).
   \end{equation}

   Bregman divergences are usually defined with strictly convex $f$, in which case $D_f( \vect x, \vect y ) = 0 \Leftrightarrow \vect x = \vect y$ and $D_f( \vect x, \vect y ) \ge 0$ for every pair $\vect x$ and $\vect y$. If $f$ is strictly convex but not differentiable, a subgradient~\cite{hil93} can be used instead of the gradient $\nabla f$ in order to obtain the same properties.

   Our techniques are not limited to provide estimates for $\expect D_f\bigl( \vect A( \vect x^* ), \vect A( \vect x^\gamma ) \bigr)$ when $f$ is convex. If $f$ is not \elias{strictly}{} convex, however, the result may not be as meaningful for the purpose of regularization parameter selection because it might happen that a minimizer of $D_f( \vect x, \vect y )$ occurs when $\vect x \neq \vect y$, in which case the Bregman divergence fails to behave as a measure of separation between vectors. We discuss computationally effective implementations of the estimators and its application to parameter selection in tomographic image reconstruction. \elias{We also present numerical experiments in order to show the effectiveness of the methodology when applied to Total-Variation regularized reconstruction in computed tomography.}{}
   
   It is worth noticing that although there are generalizations of Stein's Lemma to several noise models, all such results have been, to the best of our knowledge, applied to the approximate minimization of the averaged squared error of~\eqref{eq:UPRES}\elias{, with the exception of \cite{mab21}, where the averaged Kullback-Leibler (KL) divergence is considered, for the case where the parameter to be selected is the iteration number of a likelihood maximizing algorithm. Although the work of Massa and Bevenuto~\cite{mab21} was unknown to us during the preparation of the present manuscript, our results relate to those in~\cite{mab21} in an interesting manner, as we will discuss later}{}.
   
   A generalization \elias{of the unbiased quadratic or KL predictive estimators for quantities}{} like~\eqref{eq:UBPRES} with~\eqref{eq:bregdiv} enables the use of different divergences~\cite{csi91}, which may be more appropriate to the problem at hand. Our research is motivated by the successful use of \elias{several}{} Bregman divergences in many contexts~\cite{cds01,stg10,bmd05,azw01,gbg80}. \elias{We, therefore, introduce a technique that enables the use of many different Bregman-based risk measures as of the form~\eqref{eq:UBPRES}. In the numerical experiments, besides the classic mean-squared error, we have used Itakura-Saito and Kullback-Leibler related divergences as well.}{We present numerical experiments in order to show the effectiveness of the methodology. In the numerical experiments, besides the classic mean-squared error, we have used Itakura-Saito and Kullback-Leibler related divergences as well.}

   Finally, we discuss the reasons behind the behaviour of the method under the viewpoint of the concentration of measure phenomenon. Although we do not prove any concentration inequality in this work, we do analyze what would be the consequences of such concentration phenomena to our method if they actually hold true, \elias{and present some preliminary}{as the} numerical experimentation \elias{in order to illustrate the ideas}{seems to indicate}.

   \section{Stein-Like Estimates}\label{sec:estimate}

   We will now rely on unbiased estimators for quantities of the form
   \begin{equation}\label{eq:defg_gamma}
      \expect\left[ \vect h( \vect b )^T\vect\beta \right],
   \end{equation}
   where $\vect h : \mathbb R^m \mapsto \mathbb R^m$ and $\expect \vect b = \vect \beta$. For that, \elias{the}{}  knowledge of the probabilistic laws for $\vect b$ is required. For example, manyfold application of Stein's Lemma~\cite[Lemma 2]{ste81} leads to the following, where $\vect b \sim \mathcal N( \vect\beta, \sigma^2\eye )$ denotes a vector $\vect b$ of independent random variables such that each of its components $b_i$ is normally distributed with mean $\beta_i$ and variance $\sigma^2$:
   \begin{lemma}\label{lemm:stein}
      Let $\vect\beta \in \mathbb R^m$ and $\vect b \sim \mathcal N( \vect\beta, \sigma^2\eye )$ and consider $\vect h : \mathbb R^m \mapsto \mathbb R^m$ such that $\vect h$ is weakly differentiable and, for $i \in \{1, 2, \dots, m\}$, $\expect \left | \frac{\partial h_i}{\partial b_i}( \vect b ) \right| < \infty$. Then
      \begin{equation}
         \expect\left[ \vect h( \vect b )^T( \vect b - \vect\beta ) \right] = \sigma^2\expect\left[ \sum_{i = 1}^m \frac{\partial h_i}{\partial b_i}( \vect b ) \right].
      \end{equation}
   \end{lemma}
   We will now apply this result to the nonlinear \elias{cases}{case}~\eqref{eq:model} and \eqref{eq:reg} with expected Bregman divergence $D_f$ as a risk measure. First we rewrite:
   \begin{eqnarray}
      D_f\bigl( \vect A( \vect x^* ), \vect A( \vect x^\gamma ) \bigr) &{}= f\bigl( \vect A( \vect x^* ) \bigr) - f\bigl( \vect A( \vect x^\gamma ) \bigr) - \nabla f\bigl( \vect A( \vect x^\gamma ) \bigr)^T\bigl( \vect A( \vect x^* ) - \vect A( \vect x^\gamma ) \bigr)\nonumber\\
      & {}= f\bigl( \vect A( \vect x^* ) \bigr) - f\bigl( \vect A( \vect x^\gamma ) \bigr) - \nabla f\bigl( \vect A( \vect x^\gamma ) \bigr)^T\bigl( \vect b - \vect A( \vect x^\gamma ) \bigr)\nonumber\\
      &\qquad\qquad\qquad\qquad\qquad\qquad {}+ \nabla f\bigl( \vect A( \vect x^\gamma ) \bigr)^T\bigl( \vect b - \vect A( \vect x^* ) \bigr)\nonumber\\
      & {}= f\bigl( \vect A( \vect x^* ) \bigr) - f( \vect b ) + D_f\bigl( \vect b, \vect A( \vect x^\gamma ) \bigr)\nonumber\\
      &\qquad\qquad\qquad\qquad\qquad\qquad {}+ \nabla f\bigl( \vect A( \vect x^\gamma ) \bigr)^T\bigl( \vect b - \vect A( \vect x^* ) \bigr).\label{eq:bregmanRewrite}
   \end{eqnarray}
   Then we can prove the following \elias{result.}{}
   \begin{proposition}\label{coro:gaussEstim}
      Suppose $\vect b \sim \mathcal N\bigl( \vect A( \vect x^* ), \sigma^2 \eye \bigr)$ and let $f : \mathbb R^n \to \mathbb R$, $\vect A : \mathbb R^n \mapsto \mathbb R^m$, $\vect B_\gamma : \mathbb R^m \mapsto \mathbb R^n$, $\vect x^* \in \mathbb R^n$ be given. Define $\vect x^\gamma := \vect B_\gamma( \vect b )$ and denote
      \begin{equation}\label{eq:compositeG}
         {\vect g}_\gamma := \nabla f \circ \vect A \circ \vect B_\gamma.
      \end{equation}
      Assume that $f$, $\vect A$ and $\vect B_\gamma$ are such that ${\vect g}_\gamma$ as defined in~\eqref{eq:compositeG} is weakly differentiable, $\expect f( \vect b ) < \infty$, $\expect D_f\bigl( \vect b, \vect A( \vect x^\gamma ) \bigr) < \infty$ and that for $i \in \{ 1, 2, \dots, m \}$, $\expect_{b_i} \left | \frac{\partial g_i}{\partial b_i}( \vect b ) \right| < \infty$. Then we have:
      \begin{equation}\label{eq:g-ureExact}
         \expect D_f\bigl( \vect A( \vect x^* ), \vect A( \vect x^\gamma ) \bigr) = K + \expect D_f\bigl( \vect b, \vect A( \vect x^\gamma ) \bigr) + \sigma^2\expect\left[ \sum_{i = 1}^m \frac{\partial g_i}{\partial b_i}( \vect b ) \right],
      \end{equation}
      where $K$ is a constant independent of $\gamma$.
   \end{proposition}
   \begin{proof}
      Let $K := f\bigl( \vect A( \vect x^* ) \bigr) - \expect f( \vect b )$, then computing expectations on both sides of~\eqref{eq:bregmanRewrite} we have
      \begin{eqnarray*}
         \expect D_f\bigl( \vect A( \vect x^* ), \vect A( \vect x^\gamma ) \bigr) &{} = K + \expect D_f\bigl( \vect b, \vect A( \vect x^\gamma ) \bigr) + \expect\left[ \nabla f\bigl( \vect A( \vect x^\gamma ) \bigr)^T\bigl( \vect b - \vect A( \vect x^* ) \bigr) \right]\\
         &{} = K + \expect D_f\bigl( \vect b, \vect A( \vect x^\gamma ) \bigr) + \expect\left[ {\vect g}_\gamma( \vect b )^T\bigl( \vect b - \vect A( \vect x^* ) \bigr) \right].
      \end{eqnarray*}
      Using Lemma~\ref{lemm:stein} to replace the last term on the right gives the desired result.
   \end{proof}

   The above \elias{proposition}{Proposition} shows that, from the viewpoint of obtaining an estimator to the average risk for nonlinear models and/or reconstruction techniques, applying Stein's Lemma to a more general risk function has the same difficulty of applying it to the Mean Squared Error (MSE). Furthermore, there seems to exist compelling reasons to use other risk measures~\cite{csi91}. Indeed, a varied set of Bregman divergences have been successfully used in several applications, such as principal component analysis~\cite{cds01}; on-line density estimation~\cite{azw01}; machine learning~\cite{stg10,bmd05}, and speech processing~\cite{gbg80}.

   Let us now examine the Poisson case. Suppose that $b$ is Poisson distributed with mean $\beta$. We denote this as $b \sim \mathcal P( \beta )$. Also, if $\vect b$ is a vector of random variables such that $b_i \sim \mathcal P( \beta_i )$, we simplify the notation by $\vect b \sim \mathcal P( \vect \beta )$. Now, let $b \sim \mathcal P( \beta )$ and $h : \mathbb R \mapsto \mathbb R$ be such that $\expect_b[ h( b ) ] < \infty$, then we have~\cite{pen75}:
   \begin{equation}
      \expect_b [ \beta h( b ) ]= \expect_b [ b h( b - 1 ) ].
   \end{equation}
   This equation can be used to prove the following result~\cite[Property~2]{lbu11}:
   \begin{lemma}\label{lemm:robbin}
      Let $\vect\beta \in \mathbb R^m_+$, $\vect b \sim \mathcal P( \vect\beta )$ and consider $\vect h : \mathbb R^m \mapsto \mathbb R^m$ such that for $i \in \{1, 2, \dots, m\}$, $\expect_{b_i} [ h_i( \vect b ) ] < \infty$ and $\expect\left[ \vect h( \vect b )^T\vect b \right] < \infty$. Then
      \begin{equation}
         \expect\left[ \vect h( \vect b )^T( \vect b - \vect\beta ) \right] = \expect\left[ \vect b^T\bigl( \vect h( \vect b ) - \vect h^{[-1]}( \vect b ) \bigr) \right],
      \end{equation}
      with $\vect h^{[ \xi ]}$, for $\xi \in \mathbb R$, given componentwise as
      \begin{equation}\label{eq:def[-]}
         h^{[ \xi ]}_i( \vect b ) := h_i( \vect b + \xi\vect e^i  ),
      \end{equation}
      where $\vect e^i$ denotes the $i$-th column of the $m \times m$ identity matrix.
   \end{lemma}

   Now using Lemma~\ref{lemm:robbin} instead of Lemma~\ref{lemm:stein}, we have the following result, the proof of which we omit for similarity with the proof of Proposition~\ref{coro:gaussEstim}.
   \begin{proposition}\label{coro:poissonEstim}
      Suppose $\vect b \sim \mathcal P\bigl( \vect A( \vect x^* ) \bigr)$ and let $f$, $\vect A$, $\vect x^*$, $\vect x^\gamma$, and ${\vect g}_\gamma$ be as in Proposition~\ref{coro:gaussEstim}. Assume ${\vect g}_\gamma^{[ -1 ]}$ follows the notation of~\eqref{eq:def[-]}. Further assume that $f$, $\vect A$ and $\vect B_\gamma$ are such that $\expect f( \vect b ) < \infty$, $\expect D_f\bigl( \vect b, \vect A( \vect x^\gamma ) \bigr) < \infty$, $\expect\left[ {\vect g}_\gamma( \vect b )^T\vect b \right] < \infty$ and that for $i \in \{1, 2, \dots, m\}$, $\expect_{b_i} [ g_i( \vect b ) ] < \infty$. Then we have
      \begin{equation}\label{eq:PupbreExact}
         \expect D_f\bigl( \vect A( \vect x^* ), \vect A( \vect x^\gamma ) \bigr) = K + \expect D_f\bigl( \vect b, \vect A( \vect x^\gamma ) \bigr) + \expect\left[ \vect b^T\bigl( {\vect g}_\gamma( \vect b ) - {\vect g}_\gamma^{[-1]}( \vect b ) \bigr) \right],
      \end{equation}
      where $K$ is a constant independent of $\gamma$.
   \end{proposition}

   By now, the pattern has hopefully become evident to the reader. The idea is that given an unbiased, computable from the data, estimator for
   \begin{equation}\label{eq:estimNeccessary}
      \expect\left[ \nabla f \bigl( \vect A( \vect x^\gamma ) \bigr)^T\left( \vect b - \vect A( \vect x^* ) \right) \right],
   \end{equation}
   one can straightforwardly obtain, up to a constant, an estimator for $\expect D_f\bigl( \vect A( \vect x^* ), \vect A( \vect x^\gamma ) \bigr)$ by taking~\eqref{eq:bregmanRewrite} in consideration. Practical estimators for quantities such as~\eqref{eq:estimNeccessary} exist for a variety of noise models. For example, we can mention papers~\cite{mao14,lbu11} for the mixed Poisson-Gaussian case; \cite{eld09} for the exponential family case (which includes Gaussian, Poisson, binomial, gamma and inverse Gaussian distributions), and~\cite{hav08,lan08} for elliptically distributed errors.

   \subsection{Computation of Stein-Like Estimators}

   While equations~\eqref{eq:g-ureExact} and \eqref{eq:PupbreExact} do not rely on the unattainable quantities $\vect x^*$ or $\vect A( \vect x^* )$, both pose computational difficulties. In formula~\eqref{eq:g-ureExact} there is the need to compute the partial derivatives $\partial g_i / \partial x_i$, which depend on the derivatives of the reconstruction method. In most cases, no analytical expression for these derivatives will be available. Even for linear reconstruction methods, such as the Filtered BackProjection (FBP) algorithm, analytical expressions are not available and Monte-Carlo techniques will likely be used, as in~\cite{rbu08,mao14}. These approaches take advantage of the large number of \elias{terms}{parcels} in the summation and use random vectors to \elias{}{probe the method and} estimate the trace of the Jacobian with good accuracy and relatively low computational cost. Options such as numerically approximating each of the summands by finite differences are unfeasible as they would require the solution of a large number of related problems.

   The Monte-Carlo principle we use is given by the equality~\cite{rbu08,mao14}:
   \begin{equation}\label{eq:limit}
      \expect_{\vect\omega}\left[ \lim_{\epsilon \to 0}\frac1\epsilon\vect \omega ^T \diag( \vect z )\bigl( {\vect g}_\gamma( \vect b + \epsilon\vect\omega ) - {\vect g}_\gamma( \vect b ) \bigr)\right] = \vect z^T\vect \partial{\vect g}_\gamma( \vect b ),
   \end{equation}
   where $\vect z \in \mathbb R^m$, $\vect\omega \in \mathbb R^m$ is such that $\expect_{\vect\omega} \vect\omega = \vect 0$ and $\expect_{\vect\omega}\vect\omega\vect\omega^T = \eye$, and $\vect\partial {\vect g}_\gamma$ is defined componentwise as
   \begin{equation}
      \vect\partial_i {\vect g}_\gamma = \frac{\partial g_i}{\partial b_i}.
   \end{equation}
   Therefore, if the noise model is Gaussian, one could define the following estimator
   \begin{equation}
      \text{G-UPBRE}^f_\epsilon( \gamma ) := D_f\bigl( \vect b, \vect A( \vect x^\gamma ) \bigr) + \frac{\sigma^2}{\epsilon}\vect \omega ^T \bigl( {\vect g}_\gamma( \vect b + \epsilon\vect\omega ) - {\vect g}_\gamma( \vect b ) \bigr),
   \end{equation}
   where UPBRE stands for Unbiased Predictive Bregman Risk Estimators. Thus, because of~\eqref{eq:g-ureExact}~and~\eqref{eq:limit}, we have
   \begin{equation}\label{eq:convUPBREGauss}
      \expect_{\vect b, \vect\omega}\left[ \lim_{\epsilon \downarrow 0}\text{G-UPBRE}^f_\epsilon( \gamma ) \right] = \expect D_f\bigl( \vect A( \vect x^* ), \vect A( \vect x^\gamma ) \bigr) - K,
   \end{equation}
   where $K$ does not depend on $\gamma$.

   A discussion on the selection of the discretization parameter $\epsilon$ to be used in a practical approximation of the limit inside the expectation on the \elias{left-hand side}{left-hand-side} of~\eqref{eq:limit} can be found, for example, in~\cite{rbu08,mao14}. In both references, it has been found experimentally that under weak differentiability hypothesis, the approximation
   \begin{equation}\label{eq:diffApprox}
      \frac1\epsilon\vect \omega ^T \diag( \vect z )\bigl( {\vect g}_\gamma( \vect b + \epsilon\vect\omega ) - {\vect g}_\gamma( \vect b ) \bigr) \approx \vect z^T\vect \partial{\vect g}_\gamma( \vect b )
   \end{equation}
   appears to hold consistently within a wide range of values of $\epsilon$. Interestingly, for non-differentiable regularization methods, there still seems to be a (narrower) range for $\epsilon$ where the technique provides useful results, even without theoretical backup. In~\cite{mao14}, it was also shown that the most favorable probability distribution for $\vect\omega$, in the sense that it minimizes the variance of the resulting estimator on the left-hand side of \eqref{eq:diffApprox} for vanishing $\epsilon$, is the one for which the components $\omega_i$ of $\vect\omega$ are independently distributed with $\prob( \{ \omega_i = -1 \} ) = \prob( \{ \omega_i = 1 \} ) = 1/2$.

   Now assume a Poissonian noise model. With $\vect\omega$ as before, let us then define
   \begin{equation}\label{eq:P-UPBRE}
      \text{P-UPBRE}^f_\epsilon( \gamma ) := D_f\bigl( \vect b, \vect A( \vect x^\gamma ) \bigr) + \frac1\epsilon\vect \omega ^T \diag( \vect b )\bigl( {\vect g}_\gamma( \vect b + \epsilon\vect\omega ) - {\vect g}_\gamma( \vect b ) \bigr).
   \end{equation}
   Therefore, application of~~\eqref{eq:limit} leads to
   \begin{equation}\label{eq:convUPBREPoisson}
      \expect_{\vect\omega}\left[ \lim_{\epsilon \downarrow 0}\text{P-UPBRE}^f_\epsilon( \gamma ) \right] = D_f\bigl( \vect b, \vect A( \vect x^\gamma ) \bigr) + \vect b^T\vect\partial{\vect g}_\gamma( \vect b ).
   \end{equation}
   Furthermore, notice that a first-order Taylor expansion \elias{for the last term in~\eqref{eq:PupbreExact}}{} yields
   \begin{equation}\label{eq:TaylorApproxPoisson}
      \expect\left[ \vect b^T\bigl( {\vect g}_\gamma( \vect b ) - {\vect g}_\gamma^{[-1]}( \vect b ) \bigr) \right] \approx \expect\left[ \vect b^T\vect\partial{\vect g}_\gamma( \vect b ) \right].
   \end{equation}
   Finally, computing the expectation with \elias{respect}{relation} to $\vect b$ in both sides of~\eqref{eq:convUPBREPoisson}, taking~\eqref{eq:TaylorApproxPoisson} into consideration, and then using~\eqref{eq:PupbreExact}, we get:
   \begin{eqnarray}
      \expect_{\vect b, \vect\omega}\left[ \lim_{\epsilon \downarrow 0}\text{P-UPBRE}^f_\epsilon( \gamma ) \right] \approx \expect D_f\bigl( \vect A( \vect x^* ), \vect A( \vect x^\gamma ) \bigr) - K,
   \end{eqnarray}
   which is an approximate result, unlike~\eqref{eq:convUPBREGauss}. The approximation error should be relatively small since for Poisson random variables unity perturbations are likely to be small relatively to the size of the perturbed variables. Accordingly, numerical experimentation has found the approximation~\eqref{eq:TaylorApproxPoisson} to be accurate enough for practical applications~\cite{mao14}.

   It is not the purpose of the present paper to go through all the noise models possibly covered by the technique. Instead, we will focus our experimental work in the Poissonian case, which is the dominant noise type, e.g., in emission tomography. It is necessary, however, to notice that more sophisticated circumstances, such as a combination of Poissonian and Gaussian noise models, may lead to complications in the numerical computation of unbiased Stein-like estimators other than the simple 
   \elias{first-derivative}{first derivative}  trace estimation. On the other hand, this issue has already been previously addressed within reasonable detail in the literature~\cite{mao14} and should not be too much of a concern to the practitioner.

   \section{Numerical Experimentation}

   \subsection{The Radon Transform}

   Tomography is the production of cross-sectional images of objects in a minimally invasive manner. Several techniques have been devised in order to achieve this goal, many of which are modeled via the so-called Radon Transform (RT). The RT of a function $\eta : \mathbb R^2 \to \mathbb R$, denoted as $\mathcal R[ \eta ]$ is defined as:
   \begin{equation}
      \mathcal R[ \eta ]( \theta, t ) := \int_{\mathbb R} \eta\left( t\left(\begin{smallmatrix}\cos\theta\\ \sin\theta\end{smallmatrix}\right) + s \left(\begin{smallmatrix}-\sin\theta\\ \cos\theta\end{smallmatrix}\right)\right) \mathrm ds.
   \end{equation}
   \elias{}{
   Figure~\ref{fig:rad} (adapted from~\cite{hcc14}) depicts the geometry of this definition.

   \begin{figure}
      \centering%
      \newcommand{\shepplogan}
{%
   \begin{scope}[line width=0pt]
      %Elipse preta externa:
      \path[color=black!100,draw,fill] (0,0)       ellipse (0.69 and 0.92);
      %Elipse interna maior:
      \path[color=black!20,draw,fill]  (0,-0.0184) ellipse (0.6624 and 0.874);

      %Orgãos internos:
      \path[color=black!30,draw,fill] (0,0.35)       ellipse (0.21 and 0.25);
      \path[color=black!30,draw,fill] (0,-0.1)       ellipse (0.046 and 0.046);
      \path[color=black!30,draw,fill] (-0.08,-0.605) ellipse (0.046 and 0.023);
      \path[color=black!30,draw,fill] (0,-0.606)     ellipse (0.023 and 0.023);
      \path[color=black!30,draw,fill] (0.06,-0.605)  ellipse (0.023 and 0.046);
      \path[color=black!30,draw,fill] (0.06,-0.605)  ellipse (0.023 and 0.046);

      %Pulmão direito:
      \path[color=black!0,draw,fill,xshift=0.22\grfxunit,rotate=-18] (0,0) ellipse (0.11 and 0.31);
      %Pulmão Esquerdo:
      \path[color=black!0,draw,fill,xshift=-0.22\grfxunit,rotate=18] (0,0) ellipse (0.16 and 0.41);
      \begin{scope}
         \path[clip,xshift=-0.22\grfxunit,rotate=18] (0,0)    ellipse (0.16 and 0.41);
         \path[color=black!10,draw,fill]             (0,0.35) ellipse (0.21 and 0.25);
         \path[color=black!10,draw,fill]             (0,-0.1) ellipse (0.046 and 0.046);
      \end{scope}
   \end{scope}

   %Orgão central:
   \path[color=black!30,draw,fill]    (0,0.1)  ellipse (0.046 and 0.046);
   \begin{scope}
      \path[clip]                     (0,0.1)  ellipse (0.046 and 0.046);
      \path[color=black!40,draw,fill] (0,0.35) ellipse (0.21 and 0.25);
   \end{scope}
}%
      \input{figradonteta.tex}%
      \def\tlinha{-0.6}%
      \setlength{\grftotalwidth}{0.45\columnwidth}%
      \setlength{\grfticksize}{0.5\grfticksize}%
      \small{\ }\hfill%
      \begin{grfgraphic}{%
         \def\grfxmin{-2.05}\def\grfxmax{1.5}%
         \def\grfymin{-1.5}\def\grfymax{2.05}%
      }%
         \begin{scope}[>=stealth,style=grfaxisstyle,<->]%
            %Phantom:
            \shepplogan%
            %Eixos:
            \draw (-1.5,0) -- (1.5,0);%
            \draw (0,-1.5) -- (0,1.5);%
            \foreach \i in {-1,1}%
            {
               \draw[style=grftickstyle,-] (\i\grfxunit,-\grfticksize) -- (\i\grfxunit,\grfticksize);%
               \draw[style=grftickstyle,-] (-\grfticksize,\i\grfyunit) -- (\grfticksize,\i\grfyunit);%
            }
            \begin{scope}[rotate=\teta]
               \draw (-1.5,0) -- (1.5,0);
               \foreach \i in {-1,1}
                  \draw[style=grftickstyle,-] (\i\grfxunit,-\grfticksize) -- (\i\grfxunit,\grfticksize);
               %Caminho de integração:
               \draw[dashed,dash phase=-0.005\grfyunit] (\tlinha,-1.5) -- (\tlinha,1.5);
               \fill (\tlinha,0) node[anchor=north,inner sep=\grflabelsep] {\scriptsize$t$} circle (0.025cm);
               %Perpendicular:
               \draw[-] (\tlinha\grfxunit,0.3em) -| (\tlinha\grfxunit - 0.3em,0pt);
               \fill (\tlinha\grfxunit - 0.15em,0.15em) circle (0.025cm);
            \end{scope}
            \begin{scope}[rotate=\teta,yshift=1.5\grfyunit]
               %Gráfico:
               \draw[line width=0.025cm] plot file {figradon.data};
               \draw[-]  (1.5,0) -- (-1.5,0) node[anchor=north,rotate=\teta,inner sep=\grflabelsep] {\scriptsize$t$};
               \draw[->] (0,0)    -- (0,0.7) node[anchor=west,rotate=\teta,inner sep=\grflabelsep]  {\scriptsize$\mathcal R[ \eta ](\theta,t)$};
            \end{scope}
            %Trabalho extra para evitar bug em \tikz \arc ao desenhar o ângulo:
            \def\rad{0.075}
            \FPupn{\cpt}{0.552285 \rad{} * 90 \teta{} / *}
            \path (\teta:\rad) ++(\teta - 90:\cpt) node (a) {};
            \draw[-] (0,0) -- (\rad,0) .. controls +(0,\cpt) and (a) .. (\teta:\rad) -- cycle;
            \draw[style=grftickstyle,-,rotate=\tetameio,xshift=\rad\grfxunit] (-\grfticksize,0pt) -- (\grfticksize,0pt);
            \path[rotate=\tetameio,xshift=0.3em] (\rad,0) node[anchor=west,inner sep=0pt] {\scriptsize$\theta$};
         \end{scope}
      \end{grfgraphic}\hfill%
      \begin{grfgraphic}[1.125]{%
         \grfyaxis[R]{[]-1;0;1[]}{[]\tiny$-1$;\tiny$0$;\tiny$1$[]}%
         \grfylabel{\footnotesize$t$}%
         \grfxaxis[R]{[]0;1.57;3.14[]}{[]\tiny$0$;\tiny$\frac\pi2$;\tiny$\pi$[]}%
         \grfxlabel{\footnotesize$\theta$}%
         \def\grfxmin{0}%
         \def\grfxmax{3.14}%
         \grfwindow%
      }%
         \node[anchor=north west,inner sep=0pt] at (0,1)
         {\includegraphics[width=3.14\grfxunit,height=2\grfyunit]
         {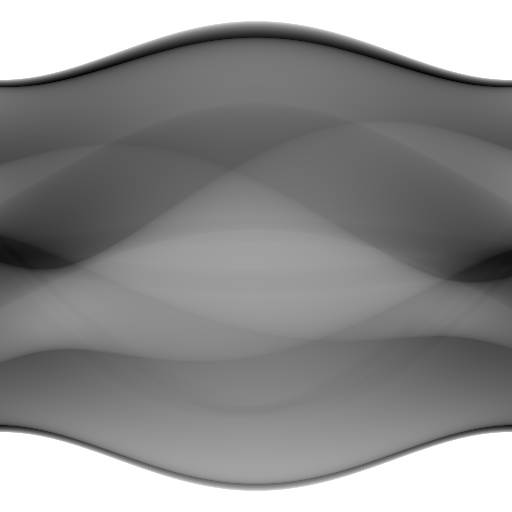}};%
      \end{grfgraphic}\hfill{\ }%
      \caption{Left: geometry of of the Radon transform. Right: Radon transform of the image shown on the
      left in the $\theta \times t$ coordinate system.}\label{fig:rad}%
   \end{figure}}

   A well known example of tomographic reconstruction technique that can be modeled with the help of the Radon transform is X-ray Computed Tomography (XCT). Experiments involving XCT are presented below. Both synthetic and real world datasets are used.

   \subsection{XCT}

   \elias{}{
   XCT is an example of \emph{transmission tomography} because the signal used to infer the RT of the imaged object is transmitted through this object of interest. The reasoning behind this process depends on the Beer-Lambert law, which states that
   \begin{equation}
      I_{\text{detected}} = I_{\text{emitted}} e^{-\int_L\mu \mathrm ds},
   \end{equation}
   where $I_{\text{detected}}$ is the expected number of photons arriving at the detector, $I_{\text{emitted}}$ is the expected number of photons departing from the source, $\int_L\mu \mathrm ds$ is the integral with respect to the arc length of the attenuation coefficient $\mu : \mathbb R^2 \to \mathbb R$ along the line $L$ connecting the source to the detector.}

   Our transmission datasets were collected at the Brazilian Synchrotron Light Laboratory (LNLS). In this kind of setup, three measurements are made for each path $L$:
   \begin{itemize}
      \item $I_{\text{dark}}( L )$: expected number of photons detected with the source turned off;
      \item $I_{\text{flat}}( L )$: expected number of photons detected with the source turned on but without object between source and detector;
      \item $N_{\text{count}}( L )$: number of photons detected with the source turned on and with the object between source and detector.
   \end{itemize}

   Photons detected during the dark scan are assumed to be part of the background radiation and are, therefore, detected in addition to the photons generated by the source. The model becomes
   \begin{equation}
      I_{\text{count}}( L ) = I_{\text{flat}}( L ) e^{-\int_L\mu \mathrm ds} + I_{\text{dark}}( L ).
   \end{equation}
   Notice that $I_{\text{flat}}( L )$ and $I_{\text{dark}}( L )$ do not involve the imaged object and can thus be estimated rather accurately. This is not true, however, for $I_{\text{count}}( L )$ and the Poisson random variable $N_{\text{count}}( L ) \sim \mathcal P\bigl( I_{\text{count}}( L ) \bigr)$ is measured instead.

   Using this principle, we can estimate the RT by
   \begin{equation}
      \mathcal R[ \mu ]( \theta_i, t_i ) \approx -\log\left( \frac{N_{\text{count}}( L_i ) - I_{\text{dark}}( L_i )}{I_{\text{flat}}( L_i )} \right),
   \end{equation}
   where $( \theta_i, t_i )$ parametrize the $i$-th line $L_i$ from source to detector according to the definition of the RT. To be precise,
   \begin{equation}
      L_i := \left\{ t_i\left(\begin{smallmatrix}\cos\theta_i\\ \sin\theta_i\end{smallmatrix}\right) + s\left(\begin{smallmatrix}-\sin\theta_i\\ \cos\theta_i\end{smallmatrix}\right) : s \in \mathbb R \right\}.
   \end{equation}

   Assuming the original image $\mu : \mathbb R^2 \to \mathbb R_+$ lies in a finite dimensional vector space generated by some basis $\{ \mu^1, \mu^2, \dots, \mu^n \}$, then it can be written as $\mu = \sum_{j = 1}^n x_j\mu^j$. Noticing, moreover, that the number of measurements is always finite in practice, one can reduce the problem of tomographic reconstruction to a linear system of equations:
   \begin{equation}
      R\vect x = \vect y,
   \end{equation}
   where $\vect x = ( x_1, x_2, \dots, x_n )^T$, the matrix $R$ is given componentwise by
   \begin{equation}
      r_{ij} = \mathcal R[\mu^j]( \theta_i, t_i ),
   \end{equation}
   and the elements $y_i$ of $\vect y$ are the corresponding Radon data, that is, $y_i = \mathcal R[ \mu ]( \theta_i, t_i )$. In practice, the above linear system of equations will be replaced by
   \begin{equation}
      R\vect x = \tilde{\vect y},
   \end{equation}
   where \elias{$\tilde{\vect y}$ is the experimentally obtained data.}{
   \begin{equation}
      {\tilde y}_i = -\log\left( \frac{N_{\text{count}}( L_i ) - I_{\text{dark}}( L_i )}{I_{\text{flat}}( L_i )} \right).
   \end{equation}}

   The set of sampled Radon coordinates $( \theta_i, t_i )$ was as follows. Let
   \begin{equation}
      T := \left\{ -1, -1 + \frac 2{2047}, -1 + 2\frac 2{2047}, -1 + 3\frac 2{2047}, \dots, -1 + 2046 \frac 2{2047}, 1 \right\}
   \end{equation}
   and
   \begin{equation}
      \Theta := \left\{ 0, \frac \pi{512}, 2\frac \pi{512}, 3\frac \pi{512}, \dots, 511\frac \pi{512} \right\}.
   \end{equation}
   Then
   \begin{equation}
      \{ ( \theta_1, t_1 ), ( \theta_2, t_2 ), \dots, ( \theta_m, t_m ) \} = \Theta \times T.
   \end{equation}

   The data acquisition was not performed directly through a photon counting sensor. Instead, a scintillator crystal~\cite{crystal} was used, which had a photon yield at $300$ K ($27^\circ$C) of $25$ Ph/KeV. The mean energy of the X-ray photons of the UVX line was around $13$ KeV. The optical coupling between the CCD detector and the scintillator achieved a $2$\% photon detection rate. Ignoring the crystal's finite temporal resolution, using an idealized monochromatic model for the light source, and ignoring the CCD's electric noise, data can be corrected dividing the CCD's visible light photon count by $6.5$ in order to estimate the X-ray photon count, which is a Poisson variable to which our methodology can be applied. We could also have used a Poissonian plus Gaussian model~\cite{lbu11} in the same way we have used the pure Poissonian model, but we wanted to keep our focus on the comparison between different Bregman divergences instead of on the noise modeling.

   \subsection{Reconstruction}

   Our discretization of the image space $[ -1, 1 ]^2$ was in a grid of $512 \times 512$ square pixels.\elias{}{ Precisely, let us define the indicator function $\chi_{S}$ of a set $S$:
   \begin{equation}
      \chi_{S}( \vect p ) := \begin{cases} 1 & \text{if } \vect p \in S \\ 0 & \text{otherwise.}\end{cases}
   \end{equation}
   Also denote the square sets, for $( i, j ) \in \{ 1, 2, \dots, 512 \}^2$, given by
   \begin{equation}
      S_{i, j} := \left[ -1 + ( j - 1 )\frac 2{512}, -1 + j\frac 2{512} \right) \times \left[ -1 + ( i - 1 )\frac 2{512}, -1 + i\frac 2{512} \right).
   \end{equation}
   Then, the basis functions are given by
   \begin{equation}
      \mu^{i, j} := \chi_{S_{i, j}},
   \end{equation}
   where we have made the lexicographic identification
   \begin{equation}
      \mu^{i, j} = \mu^{512( i - 1 ) + j}.
   \end{equation}}
   This intuitive non-overlapping basis of square pixels allows for reasonably efficient implementations of the matrix-vector products of the form $R\vect x$ and $R^T\vect y$.

   With the discretization of the problem fully defined by the basis functions and sampling scheme, we estimate the pixel values $x_i$ by solving
   \begin{equation}\label{eq:minimization}
      \min_{\vect x \in \mathbb R_+^n} \quad \frac12\| R\vect x - \tilde{\vect y} \|_2^2 + \gamma TV( \vect x ),
   \end{equation}
   with $\gamma \ge 0$ and $TV$ being the total variation
   \begin{equation}
      TV( \vect x ) := \sum_{i = 1}^n \sum_{j = 1}^n\sqrt{( x_{i, j} - x_{i, j - 1} )^2 + ( x_{i, j} - x_{i - 1, j} )^2},
   \end{equation}
   where we again have used the lexicographic identification $x_{i, j} = x_{512( i - 1 ) + j}$ and, by convention, the boundary condition $x_{0, j} = x_{i, 0} = 0$.

   We end summarizing the reconstruction method. For that, let us first denote:
   \begin{itemize}
      \item $\vect b := \bigl( N_{\text{count}}( L_1 ), N_{\text{count}}( L_2 ), \dots, N_{\text{count}}( L_{512 \times 2048} ) \bigr)^T$;
      \item $\vect f := \bigl( I_{\text{flat}}( L_1 ), I_{\text{flat}}( L_2 ), \dots, I_{\text{flat}}( L_{512 \times 2048} ) \bigr)^T$;
      \item $\vect d := \bigl( I_{\text{dark}}( L_1 ), I_{\text{dark}}( L_2 ), \dots, I_{\text{dark}}( L_{512 \times 2048} ) \bigr)^T$.
   \end{itemize}
   Then, the steps for computing $\vect B_{\gamma}( \vect b )$ are
   \begin{enumerate}
      \item Compute ${\tilde y}_i = -\log\left( \frac{b_i - d_i}{f_i} \right)$ for all $i \in \{ 1, 2, \dots, 512 \times 2048 \}$;
      \item Return the minimizer of~\eqref{eq:minimization}. The Fast Iterative Soft-Thresholding Algorithm (FISTA)~\cite{bet09b} was used to obtain the numerical minimizer.
   \end{enumerate}

   Because the input of this method is a vector of independent Poisson variables, we should be able to apply the estimator $\text{P-UPBRE}^f_\epsilon$ developed above in order to find an estimate of the optimal value for the regularization parameter $\gamma$. The final ingredient is the forward operator $\vect A$ which is given componentwise by
   \begin{equation}
      A_i( \vect x ) = d_i + f_i e^{-(R \vect x)_i},
   \end{equation}
   since this is the expected photon count over line $L_i$ for the image $\sum_{j = 1}^nx_j\mu^j$.

   \subsection{Bregman Functions}

   Throughout the numerical experimentation, we have tried three different Bregman divergences, two of which are modified versions of the Itakura-Saito \cite{fbd09} and the Kullback-Leibler \cite{kul51} divergences. These divergences are obtained using the following Bregman functions, respectively:
   \begin{equation}
      f_{\text{ms}}( \vect x ) := \sum_{i = 1}^n x_i^2,\quad f_{\text{kl}}( \vect x ) := \sum_{i = 1}^n x_i\underline\ln( x_i ),\quad\text{and}\quad f_{\text{is}}( \vect x ) := -\sum_{i = 1}^n \underline\ln( x_i ).
   \end{equation}
%
%    \begin{itemize}
%       \item $\displaystyle f_{\text{ms}}( \vect x ) := \sum_{i = 1}^n x_i^2$;
%       \item $\displaystyle f_{\text{is}}( \vect x ) := -\sum_{i = 1}^n \underline\ln( x_i )$;
%       \item $\displaystyle f_{\text{kl}}( \vect x ) := \sum_{i = 1}^n x_i\underline\ln( x_i )$.
%    \end{itemize}
   Furthermore, we define
   \begin{equation}
      \underline\ln( x ) := \begin{cases}
                              \ln( x ) & x \geq\varepsilon\\
                              \displaystyle \ln( \varepsilon ) + \frac 1\varepsilon( x - \varepsilon ) - \frac 1{2\varepsilon^2}( x - \varepsilon )^2 & x < \varepsilon
                            \end{cases}.
   \end{equation}
   We have used $\varepsilon = 10^{-1}$ in all our experiments.

   The second order approximation $\underline\ln( x )$ for the logarithm near the negative orthant was used because it is not possible to use $\ln( x )$ directly, as in the original definition of both the Itakura-Saito and Kullback-Laibler divergences, due to difficulties in the boundary of the domain of definition of the resulting functionals. The original version of the Itakura-Saito divergence, for example, reads
   \begin{equation}
      \sum_{i = 1}^n\left\{ \frac{x_i}{y_i} - \ln\left( \frac{x_i}{y_i} \right) - 1 \right\},
   \end{equation}
   which is not well defined if either $x_i = 0$ or $y_i = 0$ for some $i \in \{ 1, 2, \dots, n \}$, a common situation in, e.g., emission tomography imaging tasks, where the data will likely contain several components $b_i = 0$.

   \subsection{Real world data results}

%    \begin{figure}
%       \includegraphics[width=\columnwidth]{seed_slice_A_eps}
%       \caption{Plots of the estimator $\text{P-UPBRE}^{f_{ms}}_\epsilon( \gamma )$ over a range of values of $\gamma$ for three different values of $\epsilon$. For all of the computed values shown in this plot, the perturbation matrix $\vect w$ was kept the same.}\label{fig:est_plot_seed_slice_A}
%    \end{figure}

   \begin{figure}
      \centering%
      \includegraphics*[width=0.5\columnwidth,viewport=0.8in 0.5in 14in 9.5in]{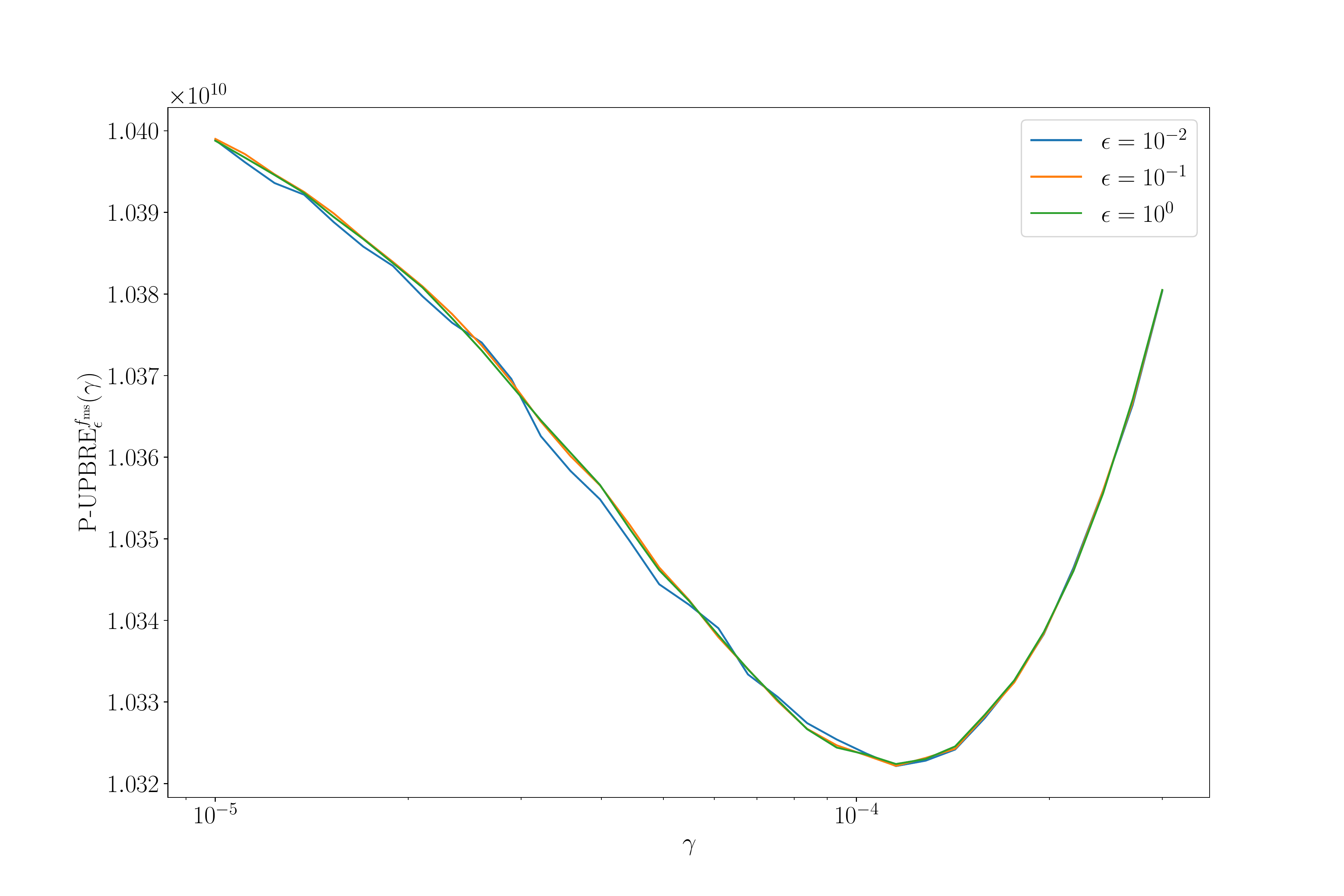}\includegraphics*[width=0.5\columnwidth,viewport=0.8in 0.5in 14in 9.5in]{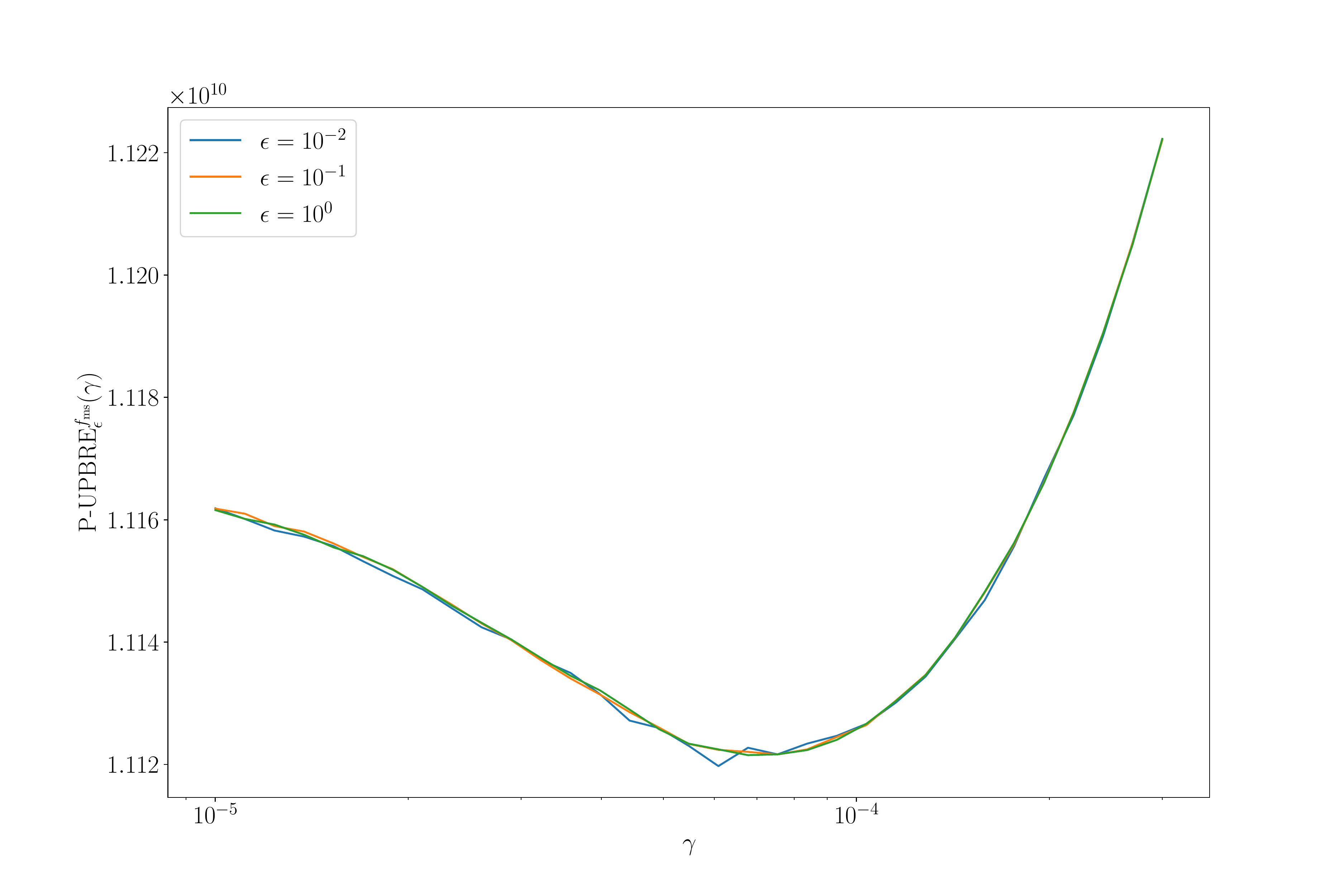}\\%
      \includegraphics*[width=0.5\columnwidth,viewport=0.8in 0.5in 14in 9.5in]{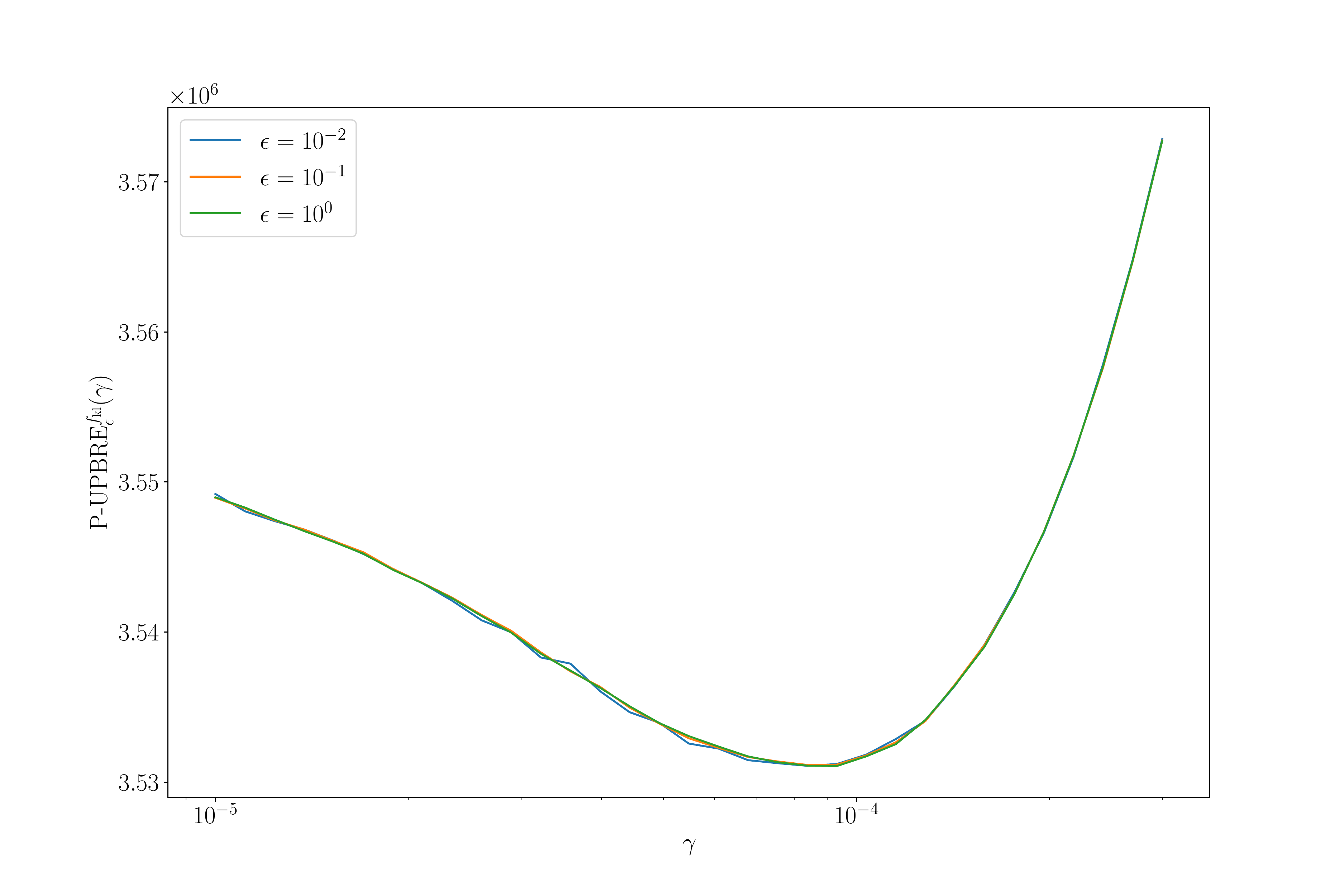}\includegraphics*[width=0.5\columnwidth,viewport=0.8in 0.5in 14in 9.5in]{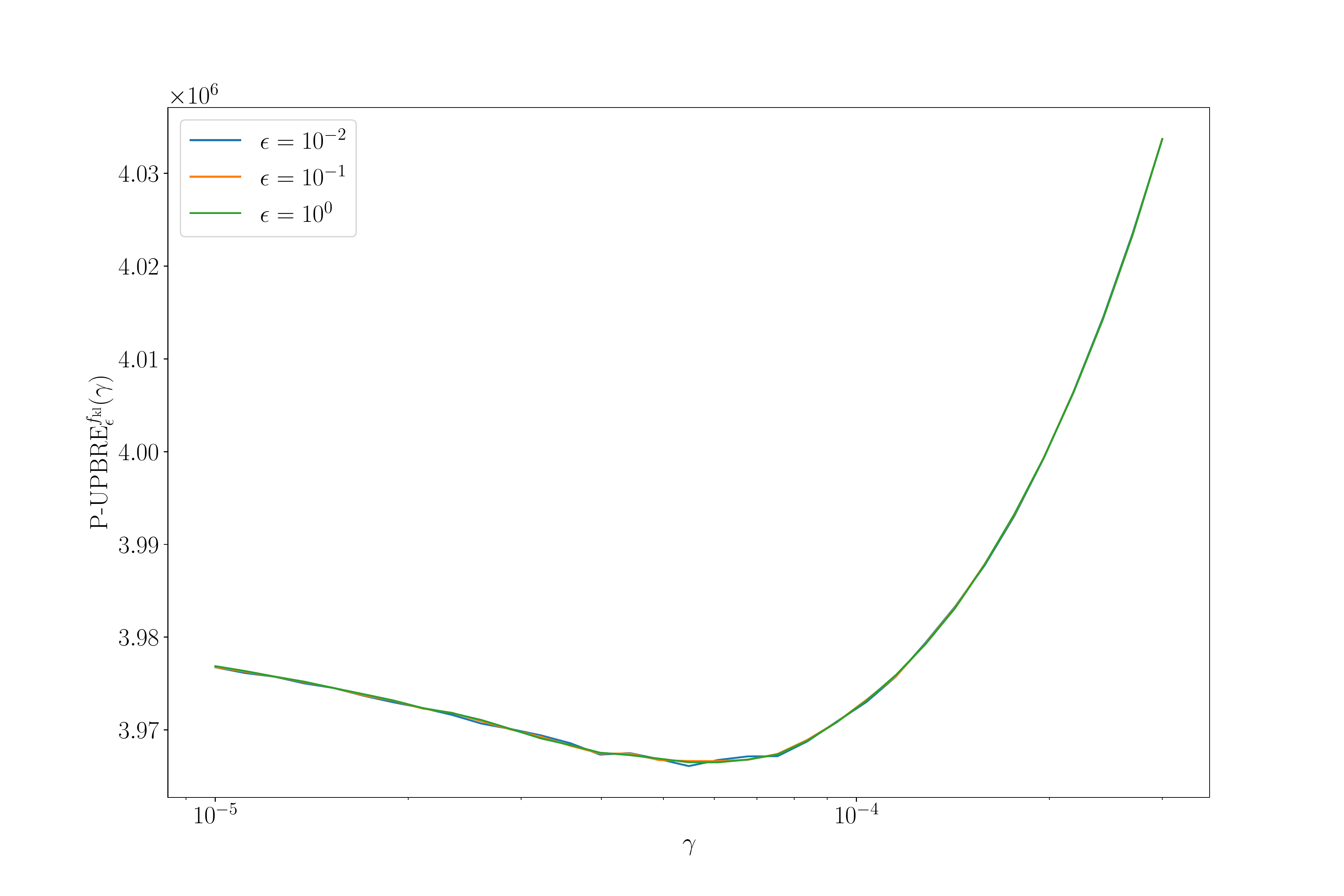}\\%
      \includegraphics*[width=0.5\columnwidth,viewport=0.8in 0.5in 14in 9.5in]{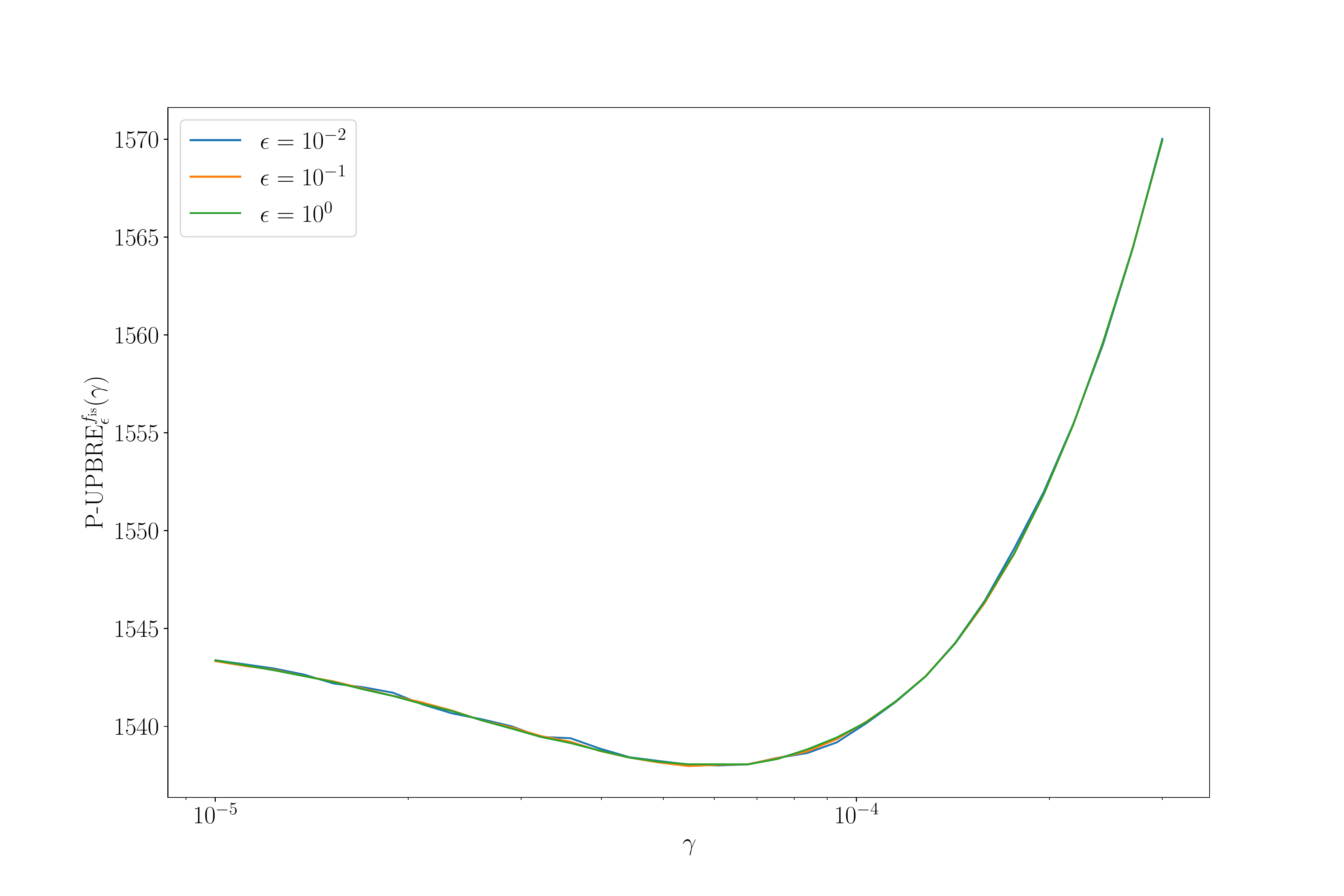}\includegraphics*[width=0.5\columnwidth,viewport=0.8in 0.5in 14in 9.5in]{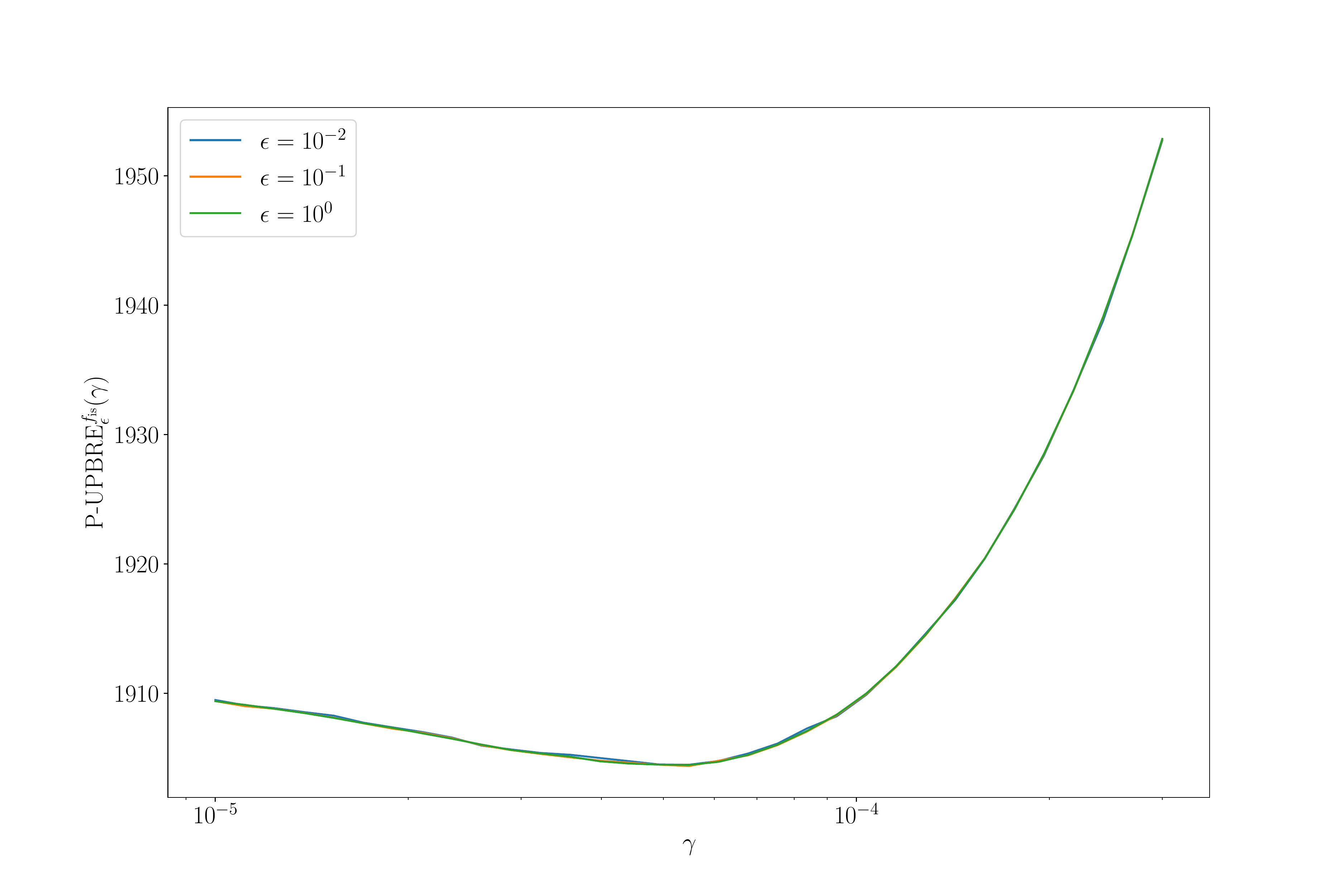}%
      \caption{Plots of the estimator $\text{P-UPBRE}^{f}_\epsilon( \gamma )$ over a range of values of $\gamma$ for three different values of $\epsilon$ and three different functions $f$. Top: $\text{P-UPBRE}^{f_{\text{ms}}}_{\epsilon}( \gamma )$. Center row: $\text{P-UPBRE}^{f_{\text{kl}}}_{\epsilon}( \gamma )$. Bottom: $\text{P-UPBRE}^{f_{\text{is}}}_{\epsilon}( \gamma )$. Left: slice shown in the left column of Figure~\ref{fig:images_seed_slices}. Right: slice shown in the center column of Figure~\ref{fig:images_seed_slices}.}
      \label{fig:seeds_plots}
   \end{figure}

   \begin{figure}
      \centering%
      \setlength{\grftotalwidth}{\textwidth}%
      \begin{grfgraphic}{%
         \def\grfxmin{-1.15}\def\grfxmax{2.0}%
         \def\grfymin{-3.5}\def\grfymax{0.5}%
%          \grfwindow%
      }%
         \pdfpxdimen=\dimexpr 1in/100\relax%
         \node[anchor=south,rotate=90,inner sep=\grflabelsep] at (-1.0,0.0) {\scriptsize No regularization\strut};%
         \node[anchor=south,rotate=90,inner sep=\grflabelsep] at (-1.0,-1.0) {\scriptsize Minimizer of P-UPBRE${}^{f_{\text{ms}}}_{10^{-1}}$};%
         \node[anchor=south,rotate=90,inner sep=\grflabelsep] at (-1.0,-2.0) {\scriptsize Minimizer of P-UPBRE${}^{f_{\text{kl}}}_{10^{-1}}$};%
         \node[anchor=south,rotate=90,inner sep=\grflabelsep] at (-1.0,-3.0) {\scriptsize Minimizer of P-UPBRE${}^{f_{\text{is}}}_{10^{-1}}$};%
         \node[inner sep=0pt] at (-0.5, 0.0) {\includegraphics[width=\grfxunit]{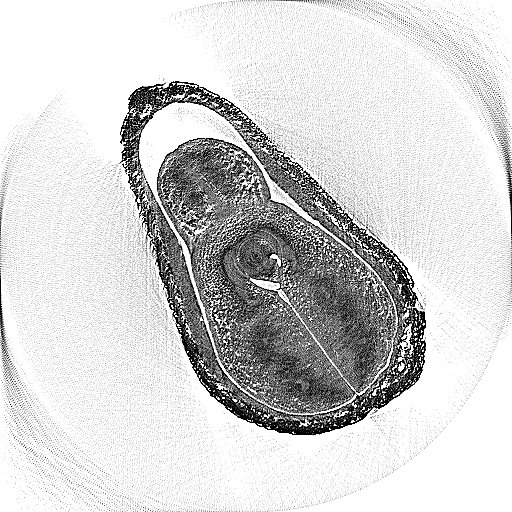}};%
         \node[inner sep=0pt] at ( 0.5, 0.0) {\includegraphics[width=\grfxunit]{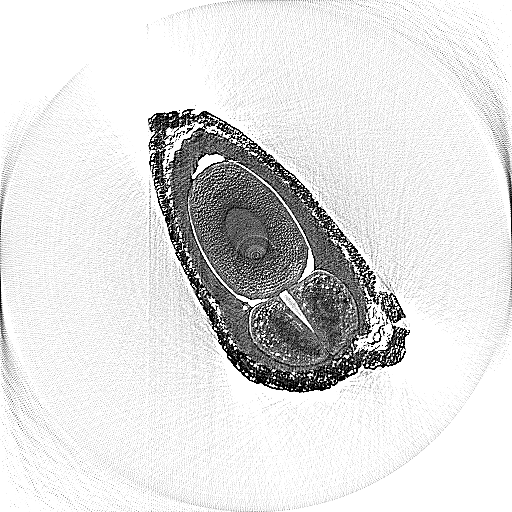}};%
         \draw[line width=0.75pt,opacity=1.0,white] ( 0.4141, -0.1836 ) rectangle ( 0.6641, 0.0664 );%
         \node[inner sep=0pt] at ( 1.5, 0.0) {\includegraphics*[viewport=212px 162px 339px 289px,width=\grfxunit]{seed_slice_A_noreg}};%
         \node[inner sep=0pt] at (-0.5,-1.0) {\includegraphics[width=\grfxunit]{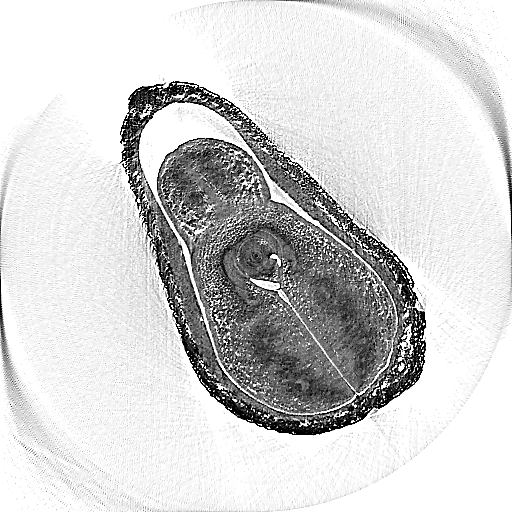}};%
         \node[inner sep=0pt] at ( 0.5,-1.0) {\includegraphics[width=\grfxunit]{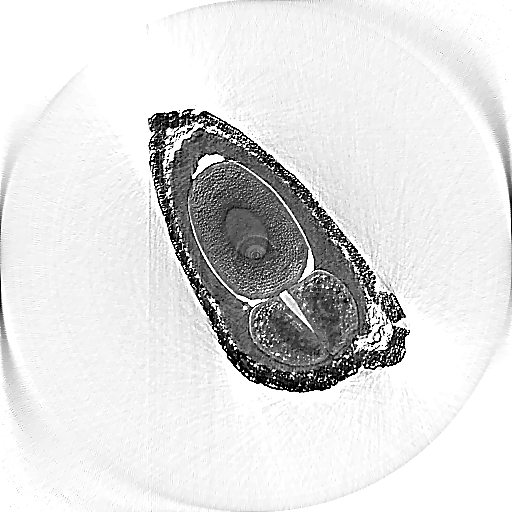}};%
         \draw[line width=0.75pt,opacity=1.0,white] ( 0.4141, -1.1836 ) rectangle ( 0.6641, -0.9336 );%
         \node[inner sep=0pt] at ( 1.5,-1.0) {\includegraphics*[viewport=212px 162px 339px 289px,width=\grfxunit]{seed_slice_A_eps_-1}};%
         \node[inner sep=0pt] at (-0.5,-2.0) {\includegraphics[width=\grfxunit]{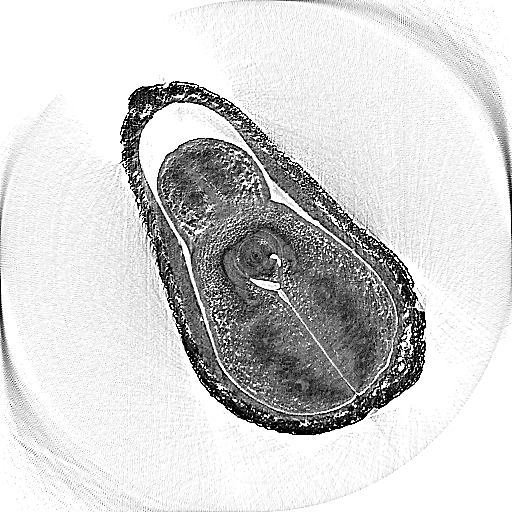}};%
         \node[inner sep=0pt] at ( 0.5,-2.0) {\includegraphics[width=\grfxunit]{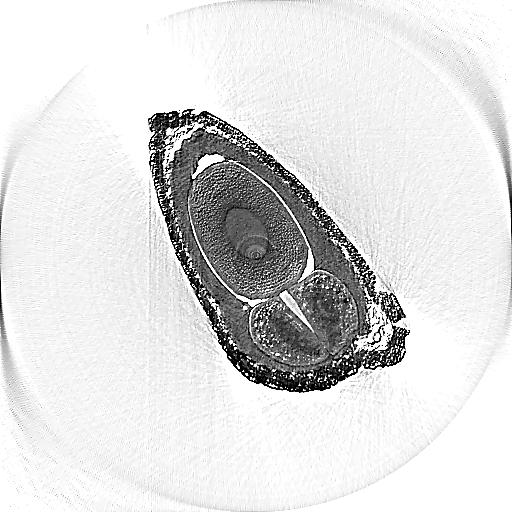}};%
         \draw[line width=0.75pt,opacity=1.0,white] ( 0.4141, -2.1836 ) rectangle ( 0.6641, -1.9336 );%
         \node[inner sep=0pt] at ( 1.5,-2.0) {\includegraphics*[viewport=212px 162px 339px 289px,width=\grfxunit]{seed_slice_A_eps_-1_kl}};%
         \node[inner sep=0pt] at (-0.5,-3.0) {\includegraphics[width=\grfxunit]{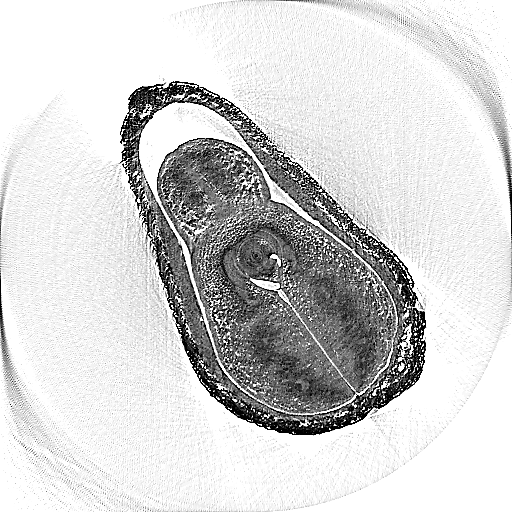}};%
         \node[inner sep=0pt] at ( 0.5,-3.0) {\includegraphics[width=\grfxunit]{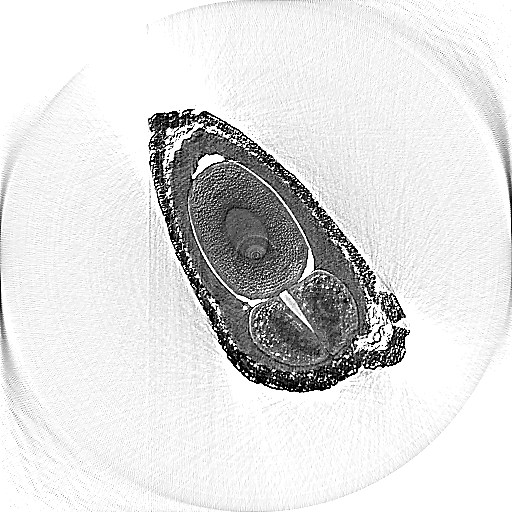}};%
         \draw[line width=0.75pt,opacity=1.0,white] ( 0.4141, -3.1836 ) rectangle ( 0.6641, -2.9336 );%
         \node[inner sep=0pt] at ( 1.5,-3.0) {\includegraphics*[viewport=212px 162px 339px 289px,width=\grfxunit]{seed_slice_A_eps_-1_is}};%
         \draw[grfaxisstyle] (-1.0,-0.5) rectangle (0.0, 0.5);%
         \draw[grfaxisstyle] ( 0.0,-0.5) rectangle (1.0, 0.5);%
         \draw[grfaxisstyle] ( 1.0,-0.5) rectangle (2.0, 0.5);%
         \draw[grfaxisstyle] (-1.0,-1.5) rectangle (0.0,-0.5);%
         \draw[grfaxisstyle] ( 0.0,-1.5) rectangle (1.0,-0.5);%
         \draw[grfaxisstyle] ( 1.0,-1.5) rectangle (2.0,-0.5);%
         \draw[grfaxisstyle] (-1.0,-2.5) rectangle (0.0,-1.5);%
         \draw[grfaxisstyle] ( 0.0,-2.5) rectangle (1.0,-1.5);%
         \draw[grfaxisstyle] ( 1.0,-2.5) rectangle (2.0,-1.5);%
         \draw[grfaxisstyle] (-1.0,-3.5) rectangle (0.0,-2.5);%
         \draw[grfaxisstyle] ( 0.0,-3.5) rectangle (1.0,-2.5);%
         \draw[grfaxisstyle] ( 1.0,-3.5) rectangle (2.0,-2.5);%
      \end{grfgraphic}
      \caption{From top to bottom: images reconstructed with no regularization; images reconstructed with the regularization parameter set as the minimizer of $\text{P-UPBRE}^{f_{\text{ms}}}_{10^{-1}}( \gamma )$ over the tested values of~$\gamma$; images reconstructed with the regularization parameter set as the minimizer of $\text{P-UPBRE}^{f_{\text{kl}}}_{10^{-1}}( \gamma )$ over the tested values of~$\gamma$; images reconstructed with the regularization parameter set as the minimizer of $\text{P-UPBRE}^{f_{\text{is}}}_{10^{-1}}( \gamma )$ over the tested values of~$\gamma$. Left and center: reconstructions of different slices of an apple seed. Right: detail of the center reconstruction.}\label{fig:images_seed_slices}
   \end{figure}

   %    \begin{figure}
%       \includegraphics[width=0.49\columnwidth]{seed_slice_A_noreg}\hfill\includegraphics[width=0.49\columnwidth]{seed_slice_A_eps_-1}
%       \caption{Left: image reconstructed without regularization. Right: image reconstructed with the regularization parameter set as the minimizer of $\text{P-UPBRE}^{f_{\text{ms}}}_{10^{-1}}( \gamma )$ over the tested values of~$\gamma$.}\label{fig:image_seed_slice_A}
%    \end{figure}

   In this subsection we will reconstruct images of slices of an apple seed scanned at the UVX tomography line of the LNLS. This first experiment is intended to assess the behaviour of the method with \elias{respect}{relation} to changes in the numerical differentiation parameter $\epsilon$ and to changes in the Bregman function. Also, it serves as a proof of concept of the methodology applied to a realistic situation.

   We first compute $\text{P-UPBRE}^f_\epsilon( \gamma )$ for $\epsilon \in \{ 10^{-2}, 10^{-1}, 10^0\}$ and for $33$ values of $\gamma$ logarithmically spaced in the range $[ 10^{-5}, 3\cdot 10^{-4} ]$. \elias{This range was selected for best visualization of the most relevant region of the domain. The value of the estimator rapidly increases outside the displayed range, which is good, for example, for numerical minimization tasks.}{} Each of the values for the discretization parameter $\epsilon$ gives rise to a curve $\bigl( \gamma, \text{P-UPBRE}^f_{\epsilon}( \gamma ) \bigr)$. In Figure~\ref{fig:seeds_plots} we see plots of these curves for $\epsilon \in \{ 10^{-2}, 10^{-1}, 10^0 \}$ (grouped in the same graphic) for all the functions $f_{\text{ms}}$, $f_{\text{kl}}$, and $f_{\text{is}}$ and for two different slices of the apple seed. There we can see that the method is reasonably robust to the choice of the numerical differentiation parameter. When this parameter becomes too small, an oscillation behavior due to numerical and floating point errors is noticeable in the curve. Tuning of the parameter can be done by gradually increasing the parameter until the oscillatory behaviour is eliminated.

   The minimizer of $\text{P-UPBRE}^f_{10^{-1}}$ should be a sound choice for the regularization parameter for the tomographic reconstruction problem. Figure~\ref{fig:images_seed_slices} shows that images reconstructed using such a minimizer as the regularization parameter indeed present a good balance between noise-removal and feature retention. Some of the artifacts seen in the images are from imperfections in the acquisition setup, such as the ring-shaped artifacts and the streaks. These are not supposed to be eliminated by the regularization. The noise, on the other hand, should be reduced. This can indeed be seen to be the case.

   A major contribution of the present paper is to generalize the idea of $\text{P-UPBRE}^{f_{\text{ms}}}_{\epsilon}$ to more general Bregman divergences. This is why we have reconstructed images from the same datasets using the minimizers of $\text{P-UPBRE}^{f_{\text{ms}}}_{\epsilon}$, $\text{P-UPBRE}^{f_{\text{kl}}}_{\epsilon}$, and $\text{P-UPBRE}^{f_{\text{is}}}_{\epsilon}$ as regularization parameters. It is possible to notice that even under the coarse sampling of the parameter space that we have used, there seems to be some noticeable, although not very large, differences among the selected regularization parameters. The next set of experiments is designed to make sense of these differences in a more controlled setting, where comparisons against a ground truth, made possible by the simulated nature of the experiment, are made.

   \subsection{Simulated data results}

   For the following set of experiments, data was simulated. The dark and flat scans were taken from the first seed slice dataset described above, but the count was simulated as a random Poisson variable:
   \begin{equation}
      N_{\text{count}}( L_i ) \sim \mathcal P\left( I_{\text{flat}}( L_i )e^{-\mathcal R[ \mu^\dagger ]( \theta_i, t_i )} + I_{\text{dark}}( L_i ) \right),
   \end{equation}
   where $\mu^\dagger$ is a $2048 \times 2048$ discretization of the Shepp-Logan~\cite{kas88} head phantom. Because the reconstruction used a $512 \times 512$ discretization, the inverse crime~\cite{kas07} is avoided.

   Each experiment consisted of generating the dataset, then minimizing P-UPBRE${}^f_\epsilon( \gamma )$ for some fixed $f$ and $\epsilon$, then minimizing $\| \vect x^\gamma - \vect x^* \|_2^2$ over $\gamma$, and then minimizing $D_f\bigl( \vect A( \vect x^* ), \vect A( \vect x^\gamma ) \bigr)$. This was repeated $20$ times for each pair $( f, \epsilon )$. The result is summarized in Figure~\ref{fig:box_minims}, where it can be seen that the regularization parameter does seem to somehow depend on the function used for the Bregman divergence.

   \begin{figure}
      \includegraphics*[viewport=0.84in 0in 9.5in 6.5in,width=0.5\textwidth]{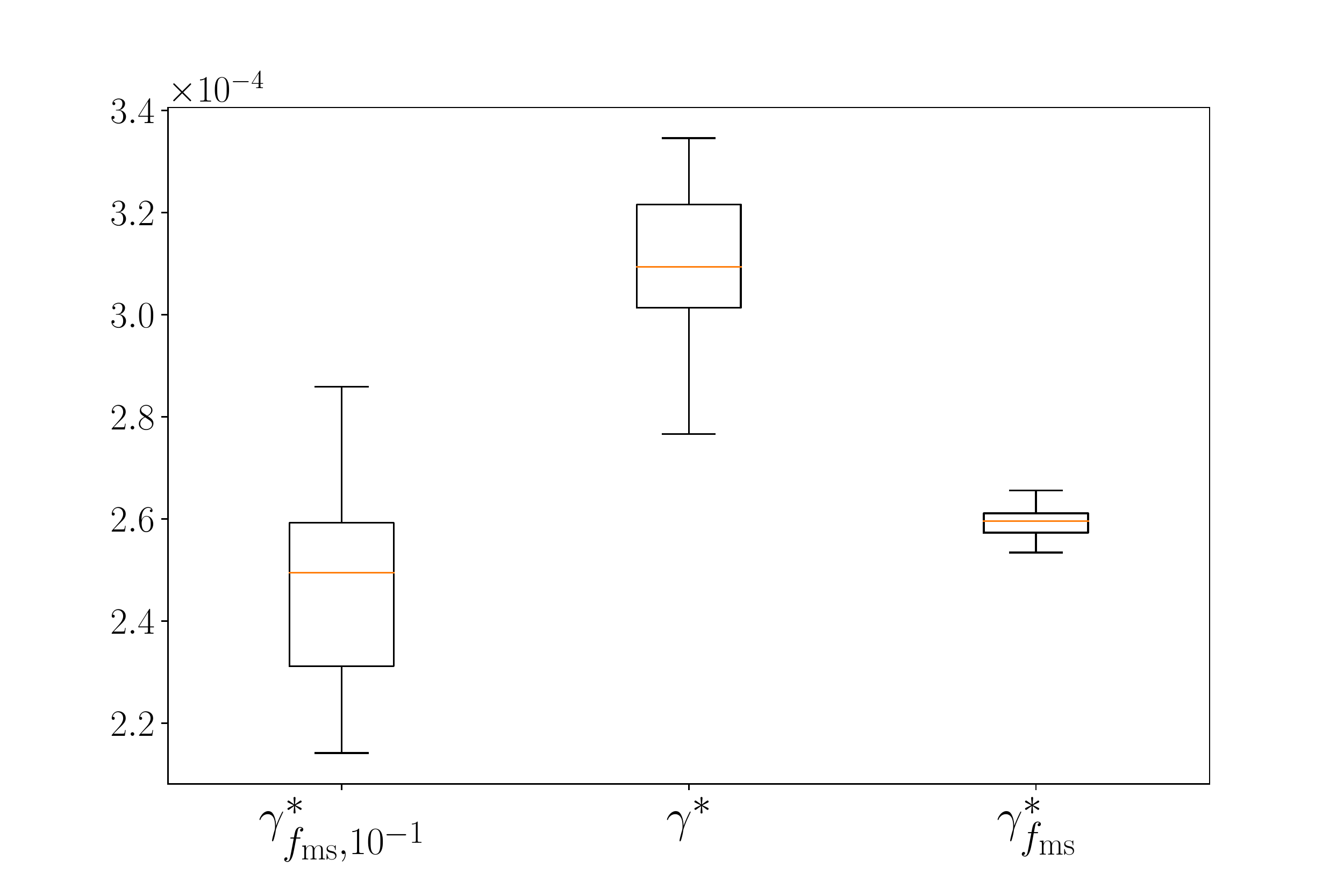}\includegraphics*[viewport=0.84in 0in 9.5in 6.5in,width=0.5\textwidth]{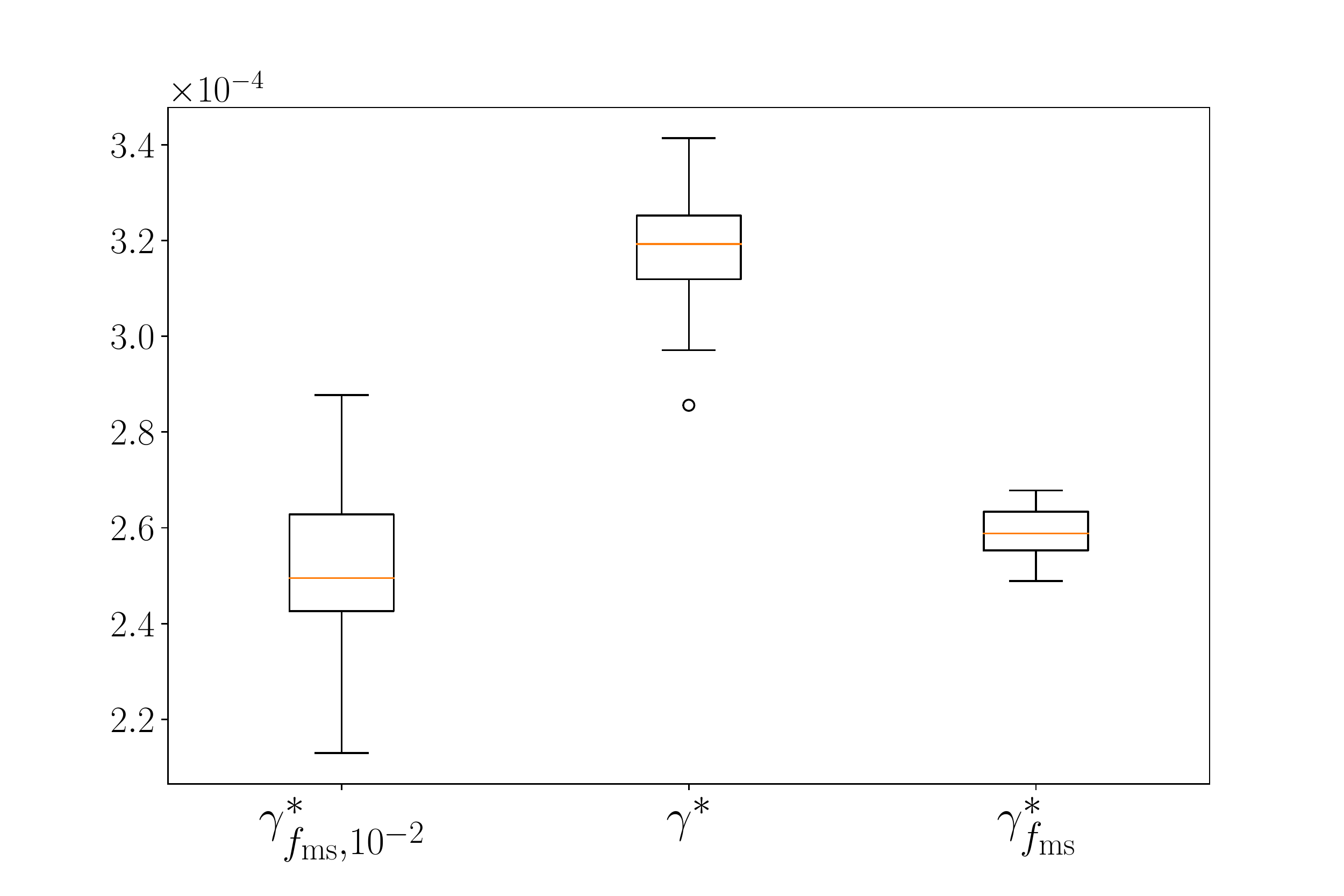}\\%
      \includegraphics*[viewport=0.84in 0in 9.5in 6.5in,width=0.5\textwidth]{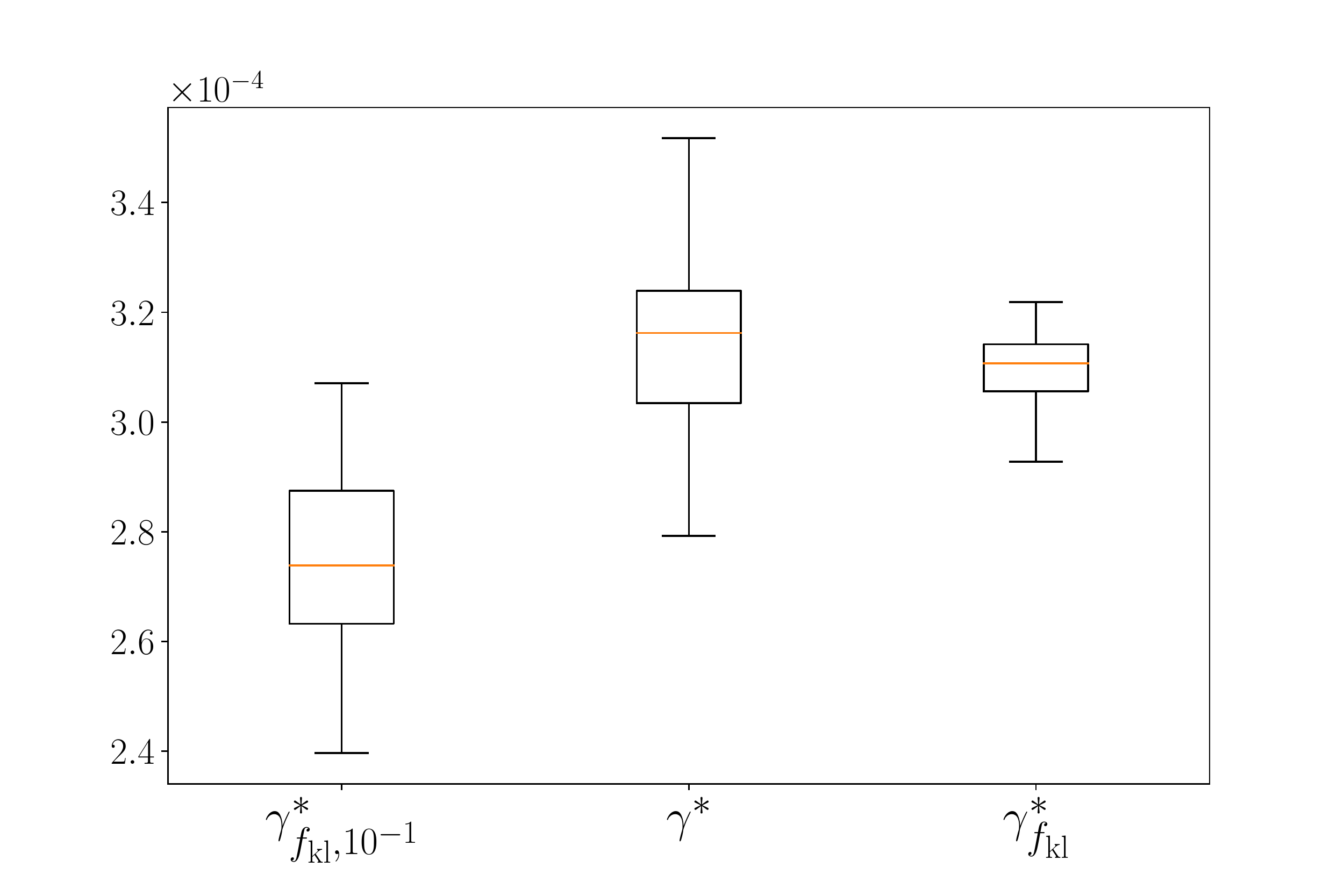}\includegraphics*[viewport=0.84in 0in 9.5in 6.5in,width=0.5\textwidth]{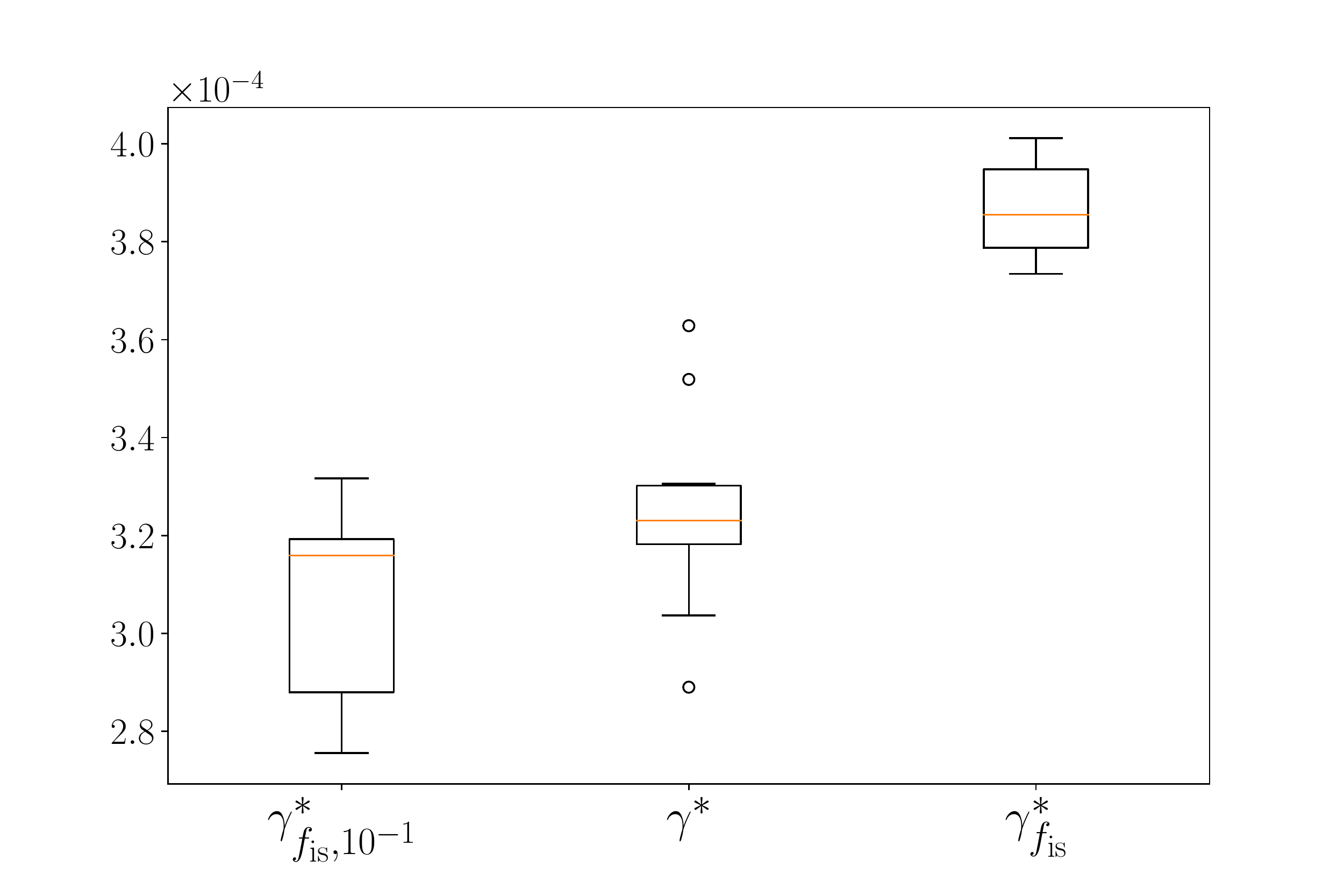}%
      \caption{Boxplots of the minimizers $\gamma^*_{f, \epsilon}$ of P-UPBRE${}^f_\epsilon( \gamma )$, the minimizers $\gamma^*$ of the squared error $\| \vect x^\gamma - \vect x^* \|_2^2$, and the minimizers $\gamma^*_f$ of the predictive Bregman error $D_f\bigl( \vect A( \vect x^* ), \vect A( \vect x^\gamma ) \bigr)$. Top left: $( f, \epsilon ) = ( f_{\text{ms}}, 10^{-1} )$. Top right: $( f, \epsilon ) = ( f_{\text{ms}}, 10^{-2} )$. Bottom left: $( f, \epsilon ) = ( f_{\text{kl}}, 10^{-1} )$. Bottom right: $( f, \epsilon ) = ( f_{\text{is}}, 10^{-1} )$.}\label{fig:box_minims}
   \end{figure}

   \begin{figure}
      \includegraphics*[viewport=0.84in 0in 9.5in 6.3in,width=0.5\textwidth]{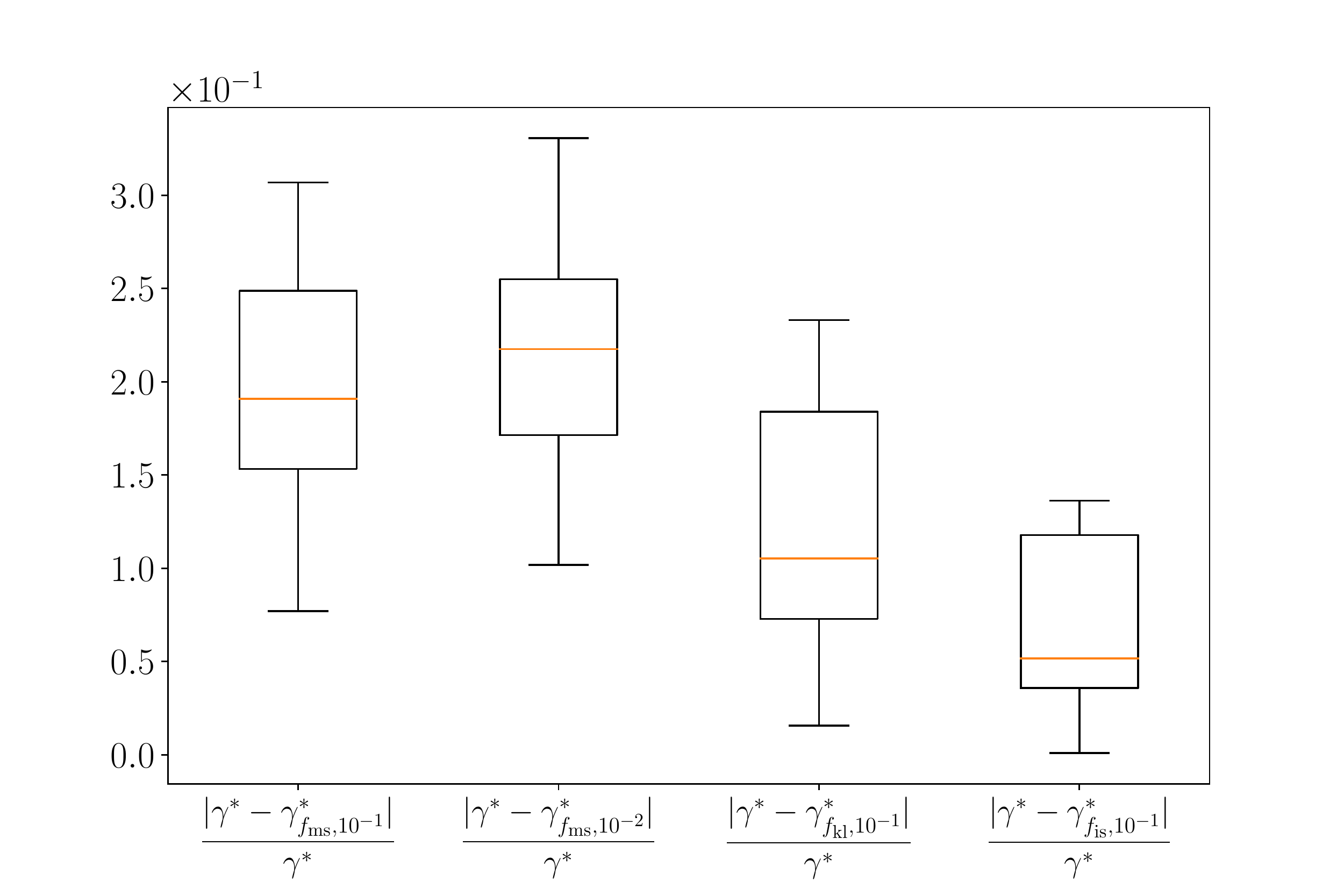}\includegraphics*[viewport=0.84in 0in 9.5in 6.3in,width=0.5\textwidth]{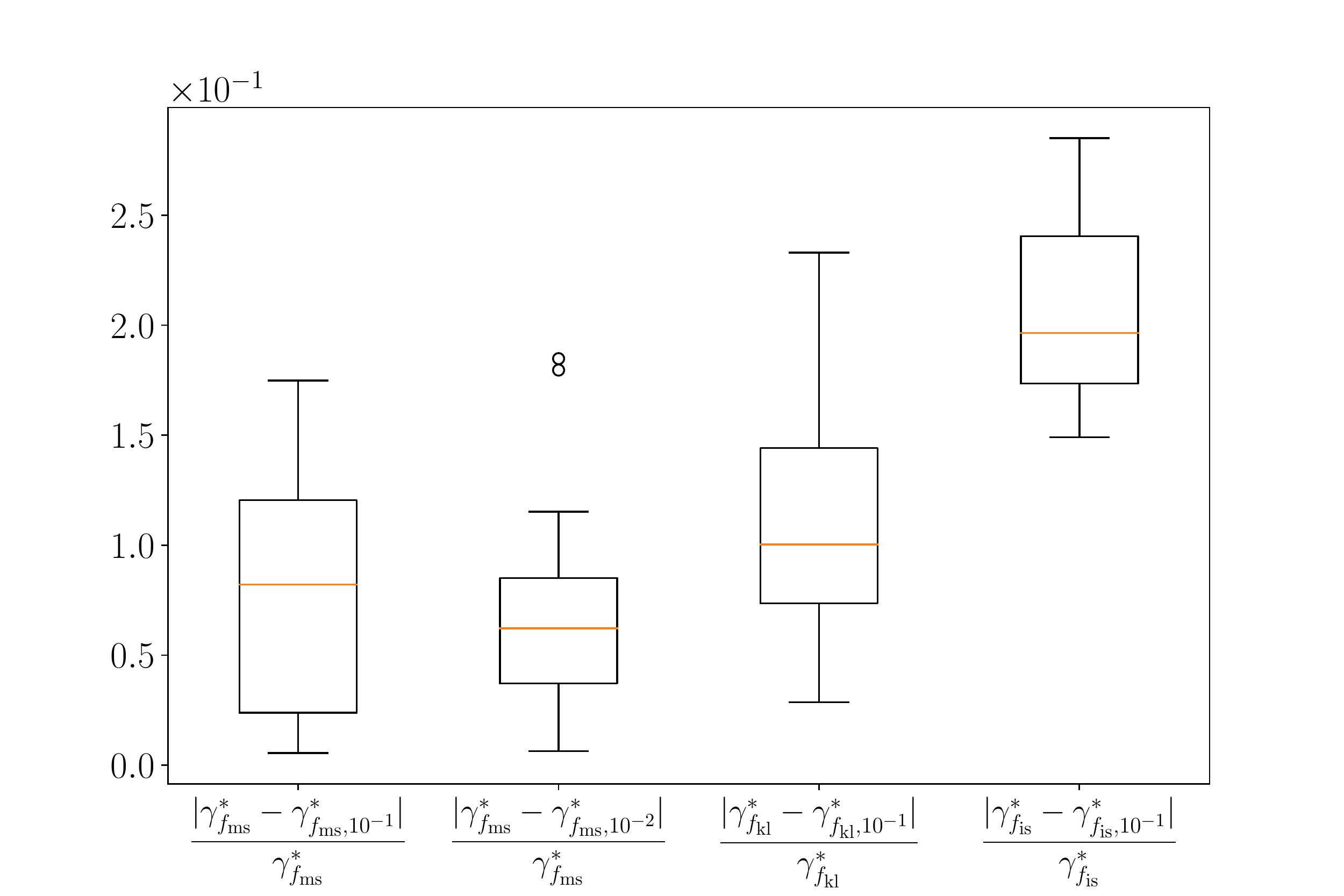}%
      \caption{Left: relative differences $| \gamma^* - \gamma^*_{f, \epsilon} | / \gamma^*$ between each minimizer $\gamma^*_{f, \epsilon}$ of P-UPBRE${}^f_\epsilon( \gamma )$ and the respective minimizer $\gamma^*$ of the squared error $\| \vect x^\gamma - \vect x^* \|_2^2$. Right: relative differences $| \gamma^*_{f} - \gamma^*_{f, \epsilon} | / \gamma^*_{f}$ between each minimizer $\gamma^*_{f, \epsilon}$ of P-UPBRE${}^f_\epsilon( \gamma )$ and the respective minimizer $\gamma^*_{f}$ of the predictive Bregman error $D_f\bigl( \vect A( \vect x^* ), \vect A( \vect x^\gamma ) \bigr)$.}\label{fig:box_errors}
   \end{figure}

   For ease of exposition, let us denote
   \begin{multline*}
      \gamma^* := \argmin_{\gamma} \| \vect x^* - \vect x^\gamma \|, \quad \gamma^*_f := \argmin_{\gamma} D_f( \vect x^*, \vect x^\gamma ), \\\text{and}\quad \gamma^*_{f, \epsilon} := \argmin_\gamma \text{P-UPBRE}^f_\epsilon( \gamma ).
   \end{multline*}
   \elias{Notice that these are minimization problems in one variable. Moreover, because we are performing simulated experiments where $\vect x^*$ is known, the objective function is easily computable. Therefore, the approximation of the above minimizers is not a complicated task and we have used the one-dimensional improved golden-section method as implemented in the {\texttt {minimize\_scalar}} routine of the SciPy package}{}.
   
   These minimizers happen to be random variables that concentrate rather tightly around its median, as shown in the experiments. We have repeated the simulated data generation $20$ times for each pair
   \begin{equation}
      ( f, \epsilon ) \in \{ ( f_{\text{ms}}, 10^{-1} ), ( f_{\text{ms}}, 10^{-2} ), ( f_{\text{kl}}, 10^{-1} ), ( f_{\text{is}}, 10^{-1} ) \}
   \end{equation}
   and we numerically computed $\gamma^*$, $\gamma^*_f$, and $\gamma^*_{f, \epsilon}$ for each of these simulated datasets. Figure~\ref{fig:box_minims} brings boxplots of the minimizers. We can observe that in fact there is not much variation, which is a useful property as we shall see in the next section.

   For now we would like to focus on the fact that the minimizers $\gamma^*_{f, \epsilon}$ are consistently below the minimizers $\gamma^*_f$, which is not surprising as the numerical differentiation scheme will affect the smoothness of the finite difference part of the estimator. Noticeably, in this example $\gamma^*_{f_{\text{is}}}$ tends to be larger than $\gamma^*$, which compensates for this fact and makes $\gamma^*_{f_{\text{is}}}$ the most accurate of the estimators for the optimal regularization parameter $\gamma^*$ as can be seen on the left of Figure~\ref{fig:box_errors}. This is so even though the difference $|\gamma^*_{f_{\text{ms}}} -  \gamma^*_{{f_{\text{ms}}}, 10^{-1}} |$ is larger than $|\gamma^*_f -  \gamma^*_{f, 10^{-1}} |$ for the other $f$, which can be seen on the right of Figure~\ref{fig:box_errors}.

   In order to verify these results with other images, discretization and noise setups, we have reconstructed a mouse head slice from the high-resolution tomographic atlas published in~\cite{mtz21}. This was done with data simulated in three different ways, two of which purposely not trying to avoid the inverse crime. The other reconstruction of the mouse head slice used the same setup as before with the Shepp-Logan phantom. Finally, we have also reconstructed the Shep-Logan phantom, this time simulating noisier data.

   The results can be seen in Figure~\ref{fig:box_minims_outros}, where the minimization of the estimator based on the Itakura-Saito divergence seems to consistently have a slightly superior accuracy. The only exception might be a statistical artifact of the small sample size of $10$ simulations for each combination of noise level, image, and discretization. Figure~\ref{fig:images_other_slices} shows some of the reconstructions that were obtained by the minimization of some of the proposed estimators. Figure~\ref{fig:images_other_slices_ideal} shows the original images we chose to reconstruct and the ``best'' reconstruction.

   Notice that in the case of the mouse head, the reconstruction obtained using $\gamma$ that minimizes $\| \vect x^* - \vect x^\gamma \|_2^2$ is noticeably smoother than those obtained minimizing P-UPBRE${}^{f}_{\epsilon}$. The reason for this seems to be that the image $\vect x^*$ used in the comparison is in fact a smoothed version of $\vect x^\dagger$, the one used to generate the data, because the resolution of $\vect x^*$ is lower than the resolution of $\vect x^\dagger$. This means that the model discrepancy confuses the estimator, which takes only the noise model into consideration. It appears that our technique allows for some of the systematic error introduced by the model inaccuracy to be reduced by minimizing a different Bregman divergence instead of the squared norm.
   
   \elias{It is interesting to observe that our methodology generalizes the approach of \cite{mab21} in the sense that if we plug the $KL$ divergence in our general approach we get the same estimator obtained in \cite{mab21}, both before and after applying the Monte-Carlo procedure (which we took from~\cite{lbu11}). It should be noted, however, that  stronger claims are proven in~\cite{mab21} about the particular regularization technique being used, which provide further insight on the nature of the approximation. Indeed, they have shown that, as the Poisson parameters increase (i.e., the relative noise level decreases), approximation~\eqref{eq:TaylorApproxPoisson} becomes more accurate for that particular regularization technique (early stopping of the EM algorithm).}{}
   
   \elias{Finally, we remark that the experiments we have presented regarding transmission tomography are unique because they deal with a model where the acquired data is a nonlinear function of the Poisson variables. Our experiments with simulated and real data show that predictive error-based are flexible and robust enough to cope with such circumstances.}{}

   \begin{figure}
      \centering%
      \setlength{\grftotalwidth}{\textwidth}%
      \begin{grfgraphic}{%
         \def\grfxmin{-1.15}\def\grfxmax{2.0}%
         \def\grfymin{-1.5}\def\grfymax{0.5}%
%          \grfwindow%
      }%
         \pdfpxdimen=\dimexpr 1in/100\relax%
         \node[anchor=south,rotate=90,inner sep=\grflabelsep] at (-1.0,0.0) {\scriptsize Original image\strut};%
         \node[anchor=south,rotate=90,inner sep=\grflabelsep] at (-1.0,-1.0) {\scriptsize Minimizer of $\| \vect x^* - \vect x^\gamma \|$};%
         \node[inner sep=0pt] at (-0.5, 0.0) {\includegraphics[width=\grfxunit]{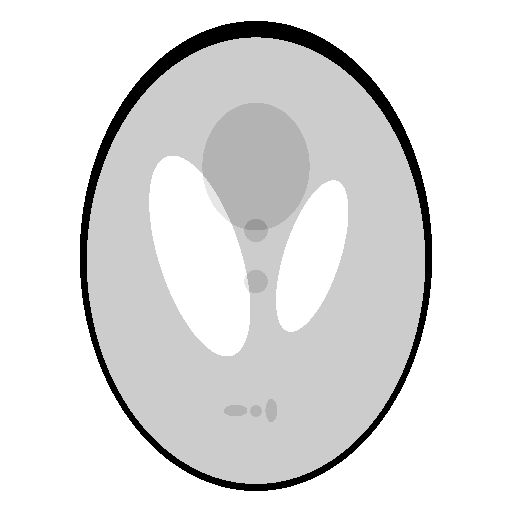}};%
         \node[inner sep=0pt] at ( 0.5, 0.0) {\includegraphics[width=\grfxunit]{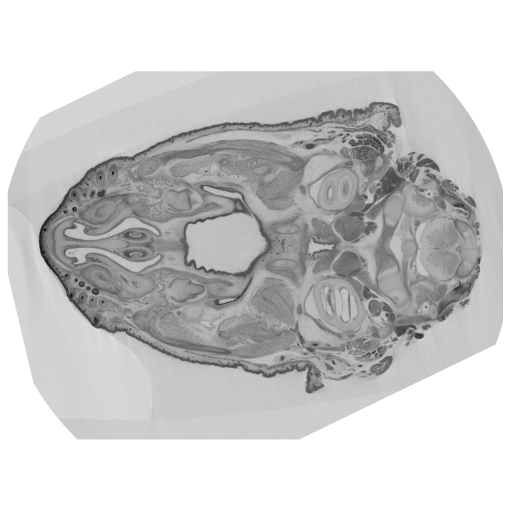}};%
         \draw[line width=0.75pt,opacity=1.0,white] ( 0.5859, 0.0469 ) rectangle ( 0.8359, 0.2969 );%
         \node[inner sep=0pt] at ( 1.5, 0.0) {\includegraphics*[viewport=300px 280px 427px 407px,width=\grfxunit]{img_small}};%
         \node[inner sep=0pt] at (-0.5,-1.0) {\includegraphics[width=\grfxunit]{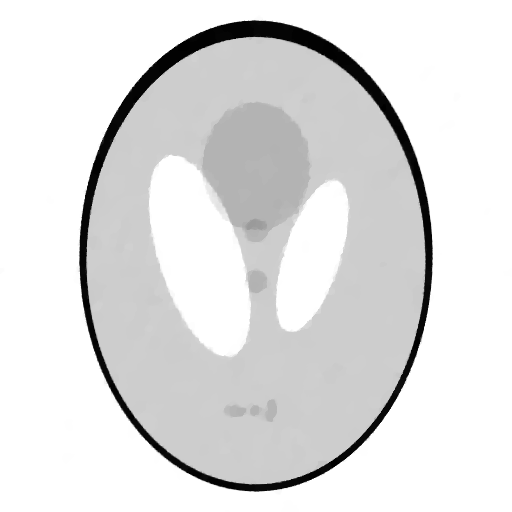}};%
         \node[inner sep=0pt] at ( 0.5,-1.0) {\includegraphics[width=\grfxunit]{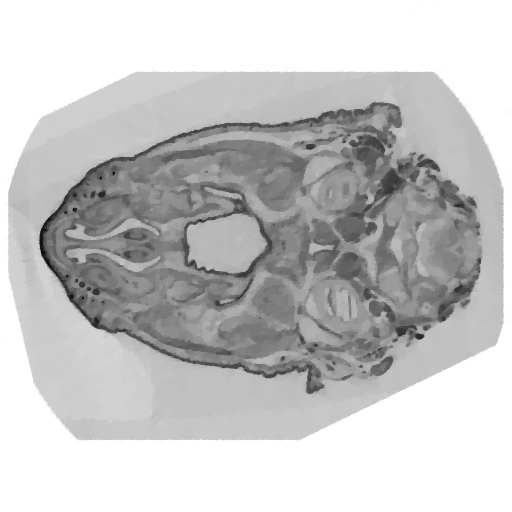}};%
         \draw[line width=0.75pt,opacity=1.0,white] ( 0.5859, -0.9531 ) rectangle ( 0.8359, -0.7031 );%
         \node[inner sep=0pt] at ( 1.5,-1.0) {\includegraphics*[viewport=300px 280px 427px 407px,width=\grfxunit]{results_all_kappa65_rato_ideal}};%
         \draw[grfaxisstyle] (-1.0,-0.5) rectangle (0.0, 0.5);%
         \draw[grfaxisstyle] ( 0.0,-0.5) rectangle (1.0, 0.5);%
         \draw[grfaxisstyle] ( 1.0,-0.5) rectangle (2.0, 0.5);%
         \draw[grfaxisstyle] (-1.0,-1.5) rectangle (0.0,-0.5);%
         \draw[grfaxisstyle] ( 0.0,-1.5) rectangle (1.0,-0.5);%
         \draw[grfaxisstyle] ( 1.0,-1.5) rectangle (2.0,-0.5);%
      \end{grfgraphic}
      \caption{Top: original images. Bottom: images reconstructed with ``optimal'' regularization. Left: Shepp-Logan phantom. Center: mouse head. Right: detail of the center image.}\label{fig:images_other_slices_ideal}
   \end{figure}

   \begin{figure}
      \centering%
      \setlength{\grftotalwidth}{\textwidth}%
      \begin{grfgraphic}{%
         \def\grfxmin{-1.15}\def\grfxmax{2.0}%
         \def\grfymin{-3.5}\def\grfymax{0.5}%
%          \grfwindow%
      }%
         \pdfpxdimen=\dimexpr 1in/100\relax%
         \node[anchor=south,rotate=90,inner sep=\grflabelsep] at (-1.0,0.0) {\scriptsize No regularization\strut};%
         \node[anchor=south,rotate=90,inner sep=\grflabelsep] at (-1.0,-1.0) {\scriptsize Minimizer of P-UPBRE${}^{f_{\text{ms}}}_{10^{-1}}$};%
         \node[anchor=south,rotate=90,inner sep=\grflabelsep] at (-1.0,-2.0) {\scriptsize Minimizer of P-UPBRE${}^{f_{\text{kl}}}_{10^{-1}}$};%
         \node[anchor=south,rotate=90,inner sep=\grflabelsep] at (-1.0,-3.0) {\scriptsize Minimizer of P-UPBRE${}^{f_{\text{is}}}_{10^{-1}}$};%
         \node[inner sep=0pt] at (-0.5, 0.0) {\includegraphics[width=\grfxunit]{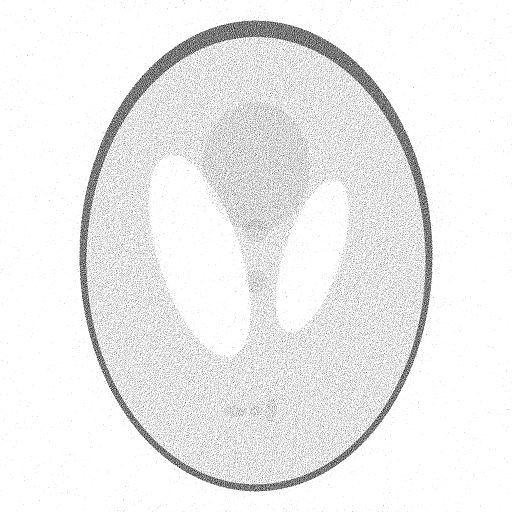}};%
         \node[inner sep=0pt] at ( 0.5, 0.0) {\includegraphics[width=\grfxunit]{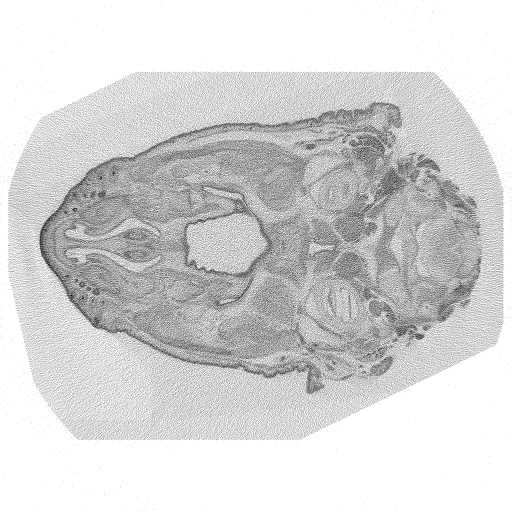}};%
         \draw[line width=0.75pt,opacity=1.0,white] ( 0.5859, 0.0469 ) rectangle ( 0.8359, 0.2969 );%
         \node[inner sep=0pt] at ( 1.5, 0.0) {\includegraphics*[viewport=300px 280px 427px 407px,width=\grfxunit]{results_all_kappa65_rato_noreg}};%
         \node[inner sep=0pt] at (-0.5,-1.0) {\includegraphics[width=\grfxunit]{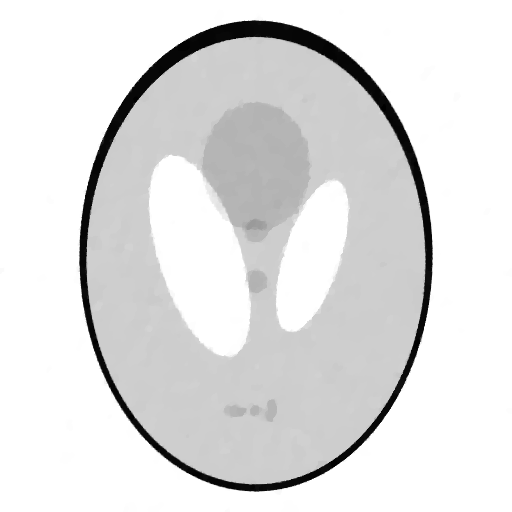}};%
         \node[inner sep=0pt] at ( 0.5,-1.0) {\includegraphics[width=\grfxunit]{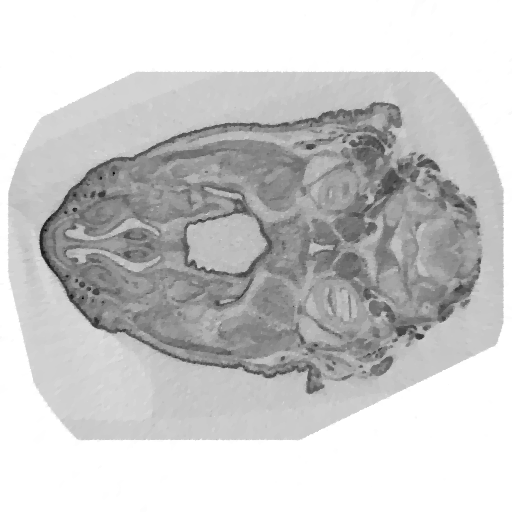}};%
         \draw[line width=0.75pt,opacity=1.0,white] ( 0.5859, -0.9531 ) rectangle ( 0.8359, -0.7031 );%
         \node[inner sep=0pt] at ( 1.5,-1.0) {\includegraphics*[viewport=300px 280px 427px 407px,width=\grfxunit]{results_all_kappa65_rato_sn}};%
         \node[inner sep=0pt] at (-0.5,-2.0) {\includegraphics[width=\grfxunit]{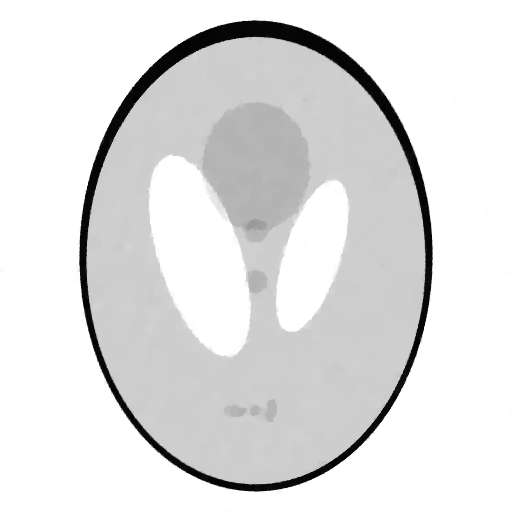}};%
         \node[inner sep=0pt] at ( 0.5,-2.0) {\includegraphics[width=\grfxunit]{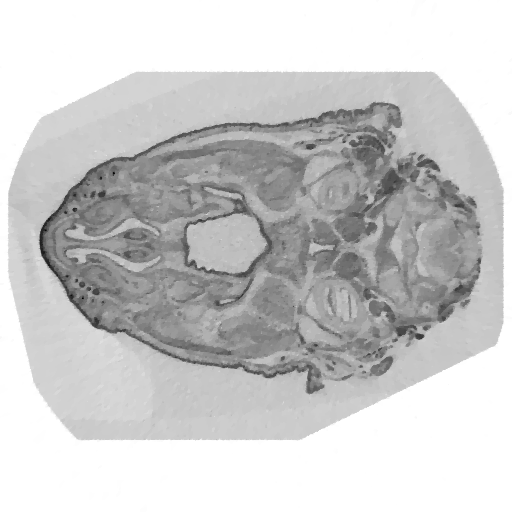}};%
         \draw[line width=0.75pt,opacity=1.0,white] ( 0.5859, -1.9531 ) rectangle ( 0.8359, -1.7031 );%
         \node[inner sep=0pt] at ( 1.5,-2.0) {\includegraphics*[viewport=300px 280px 427px 407px,width=\grfxunit]{results_all_kappa65_rato_kl}};%
         \node[inner sep=0pt] at (-0.5,-3.0) {\includegraphics[width=\grfxunit]{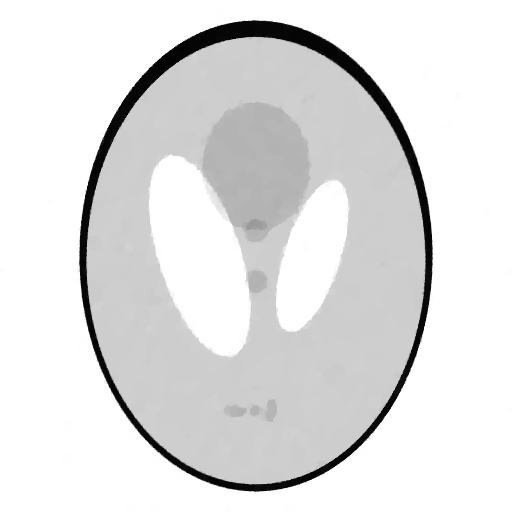}};%
         \node[inner sep=0pt] at ( 0.5,-3.0) {\includegraphics[width=\grfxunit]{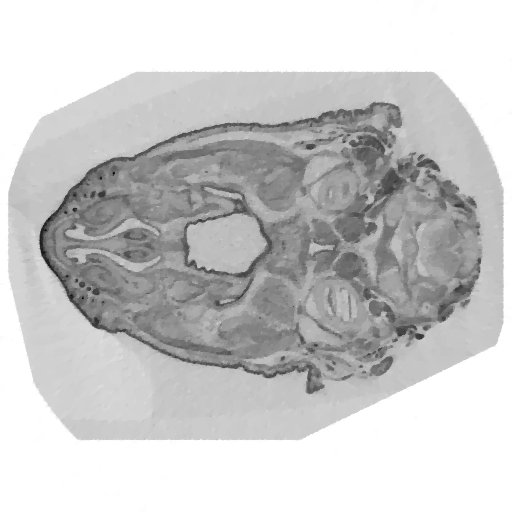}};%
         \draw[line width=0.75pt,opacity=1.0,white] ( 0.5859, -2.9531 ) rectangle ( 0.8359, -2.7031 );%
         \node[inner sep=0pt] at ( 1.5,-3.0) {\includegraphics*[viewport=300px 280px 427px 407px,width=\grfxunit]{results_all_kappa65_rato_is}};%
         \draw[grfaxisstyle] (-1.0,-0.5) rectangle (0.0, 0.5);%
         \draw[grfaxisstyle] ( 0.0,-0.5) rectangle (1.0, 0.5);%
         \draw[grfaxisstyle] ( 1.0,-0.5) rectangle (2.0, 0.5);%
         \draw[grfaxisstyle] (-1.0,-1.5) rectangle (0.0,-0.5);%
         \draw[grfaxisstyle] ( 0.0,-1.5) rectangle (1.0,-0.5);%
         \draw[grfaxisstyle] ( 1.0,-1.5) rectangle (2.0,-0.5);%
         \draw[grfaxisstyle] (-1.0,-2.5) rectangle (0.0,-1.5);%
         \draw[grfaxisstyle] ( 0.0,-2.5) rectangle (1.0,-1.5);%
         \draw[grfaxisstyle] ( 1.0,-2.5) rectangle (2.0,-1.5);%
         \draw[grfaxisstyle] (-1.0,-3.5) rectangle (0.0,-2.5);%
         \draw[grfaxisstyle] ( 0.0,-3.5) rectangle (1.0,-2.5);%
         \draw[grfaxisstyle] ( 1.0,-3.5) rectangle (2.0,-2.5);%
      \end{grfgraphic}
      \caption{From top to bottom: images reconstructed with no regularization; images reconstructed with the regularization parameter set as the numerical minimizer of $\text{P-UPBRE}^{f_{\text{ms}}}_{10^{-1}}( \gamma )$; images reconstructed with the regularization parameter set as the numerical minimizer of $\text{P-UPBRE}^{f_{\text{kl}}}_{10^{-1}}( \gamma )$; images reconstructed with the regularization parameter set as the numerical minimizer of $\text{P-UPBRE}^{f_{\text{is}}}_{10^{-1}}( \gamma )$. Left: Shepp-Logan phantom. Center: mouse head. Right: detail of the center image.}\label{fig:images_other_slices}
   \end{figure}

   \begin{figure}
      \includegraphics*[viewport=0.7in 0in 9.5in 6.5in,width=0.5\textwidth]{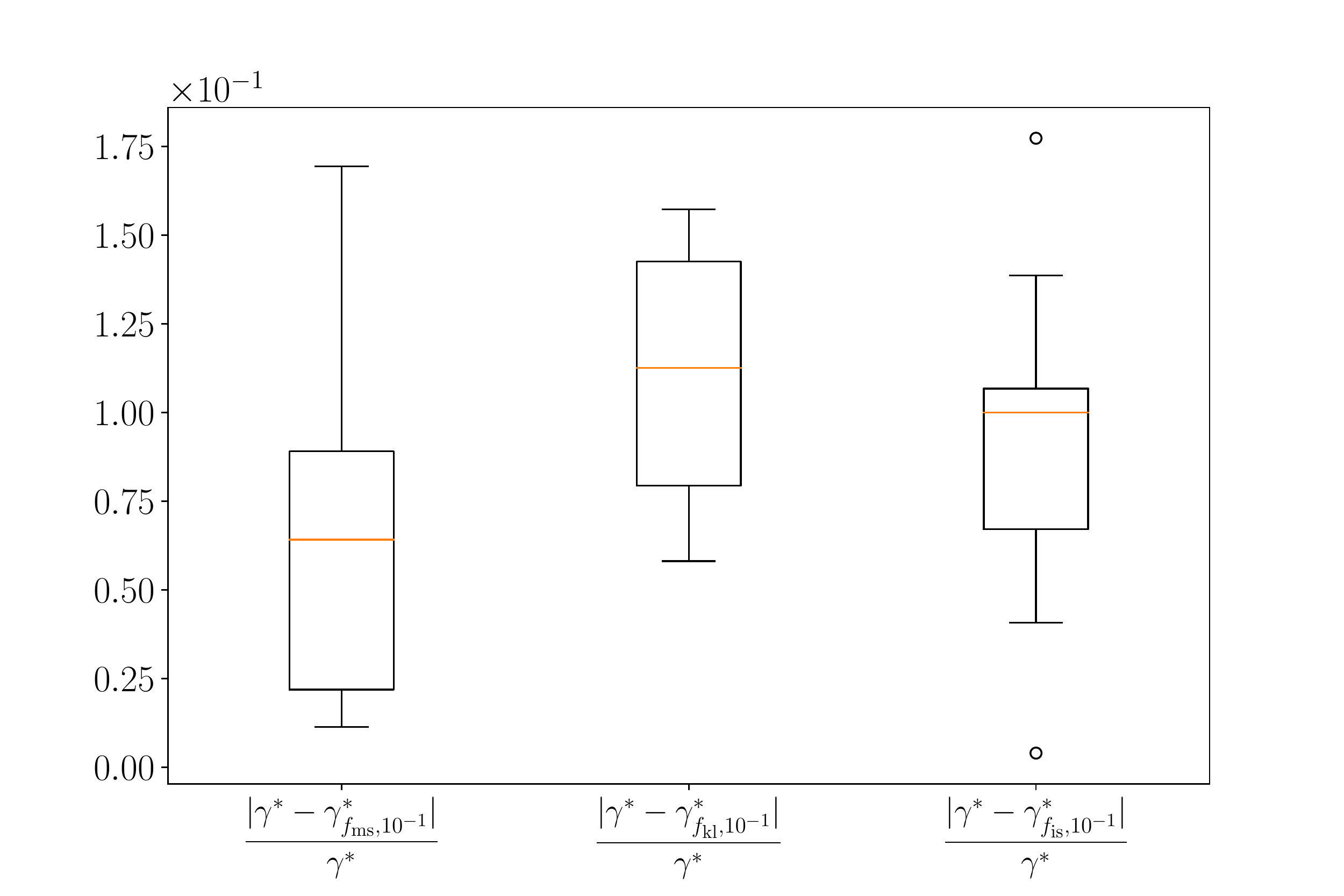}\includegraphics*[viewport=0.7in 0in 9.5in 6.5in,width=0.5\textwidth]{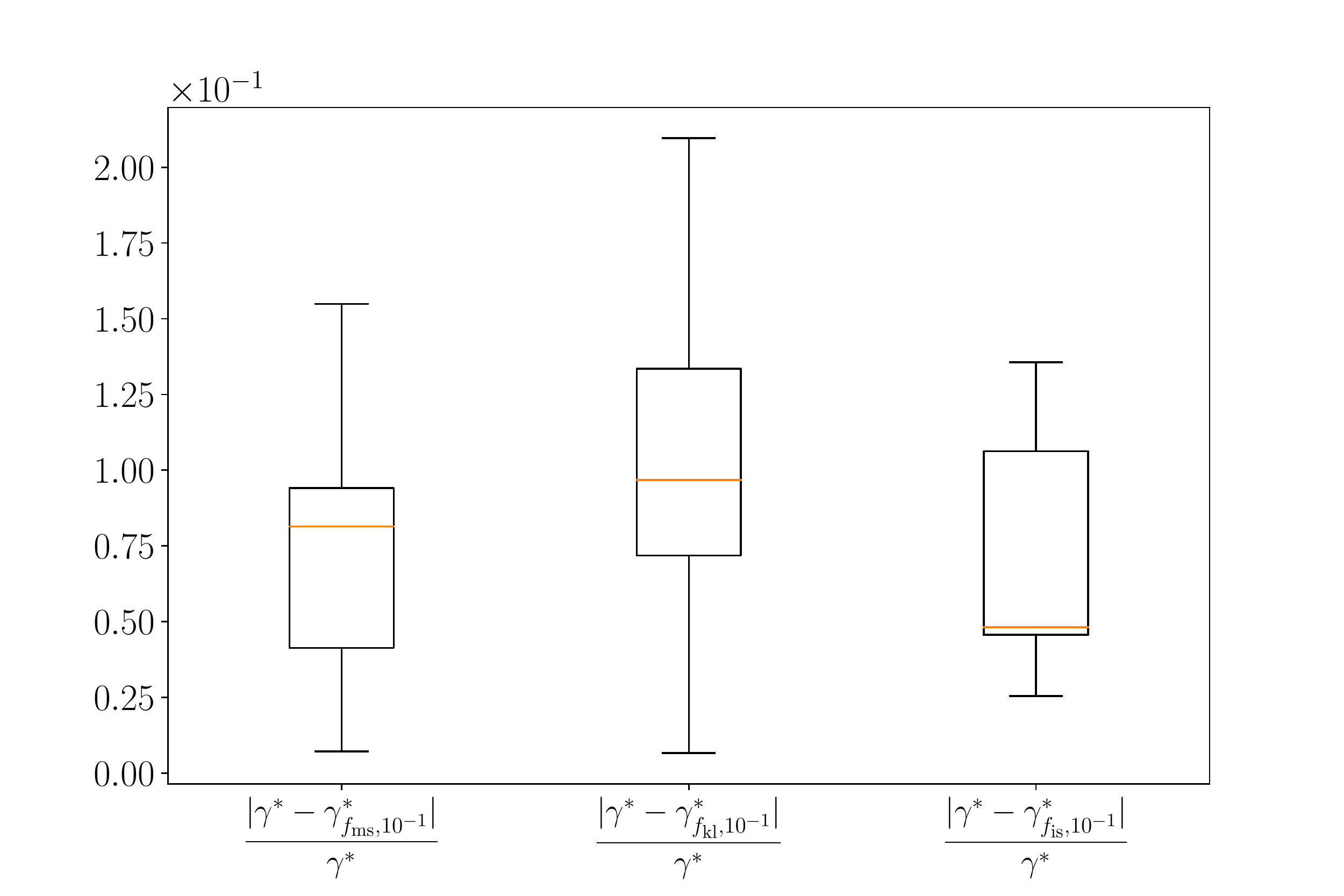}\\%
      \includegraphics*[viewport=0.84in 0in 9.5in 6.5in,width=0.5\textwidth]{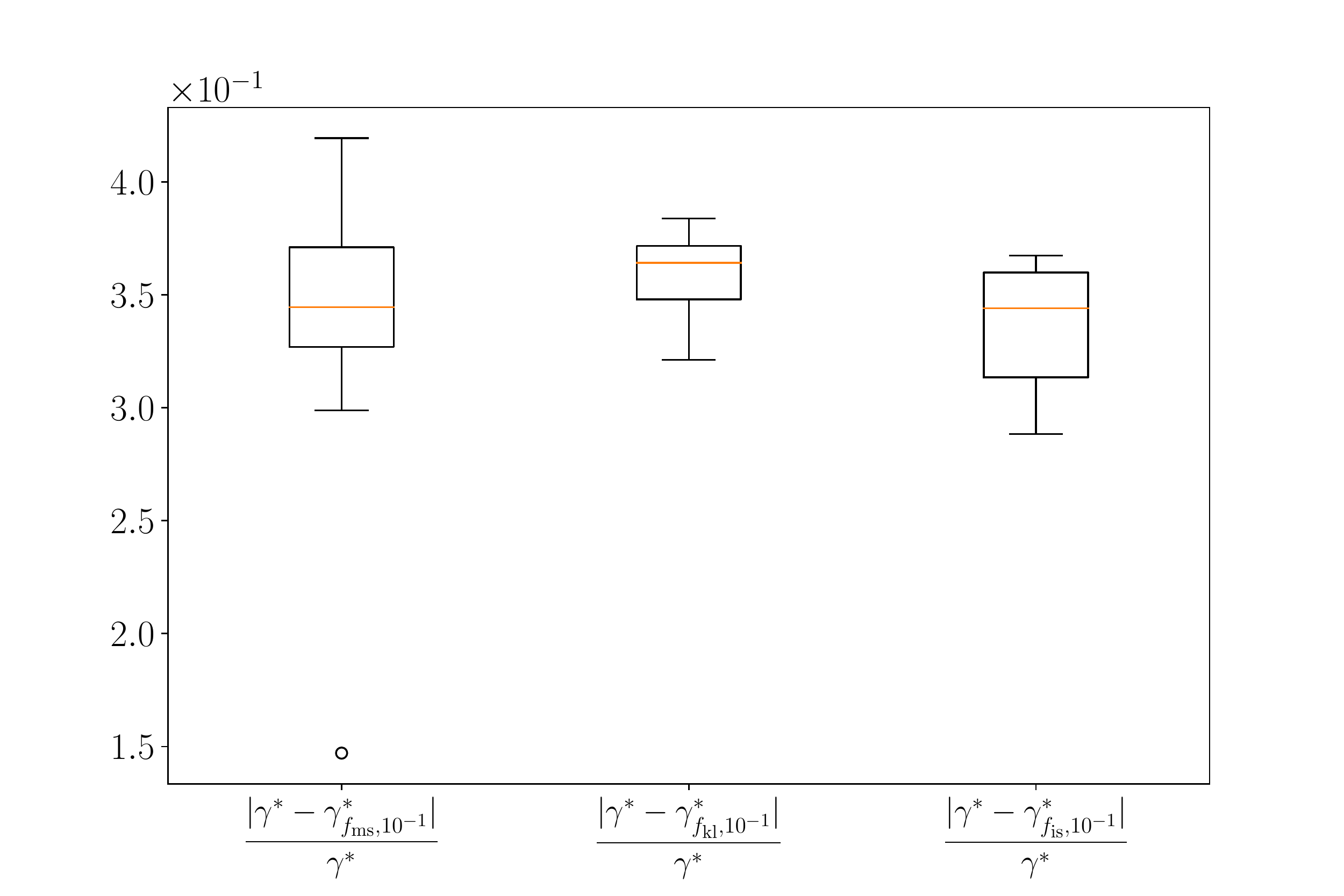}\includegraphics*[viewport=0.84in 0in 9.5in 6.5in,width=0.5\textwidth]{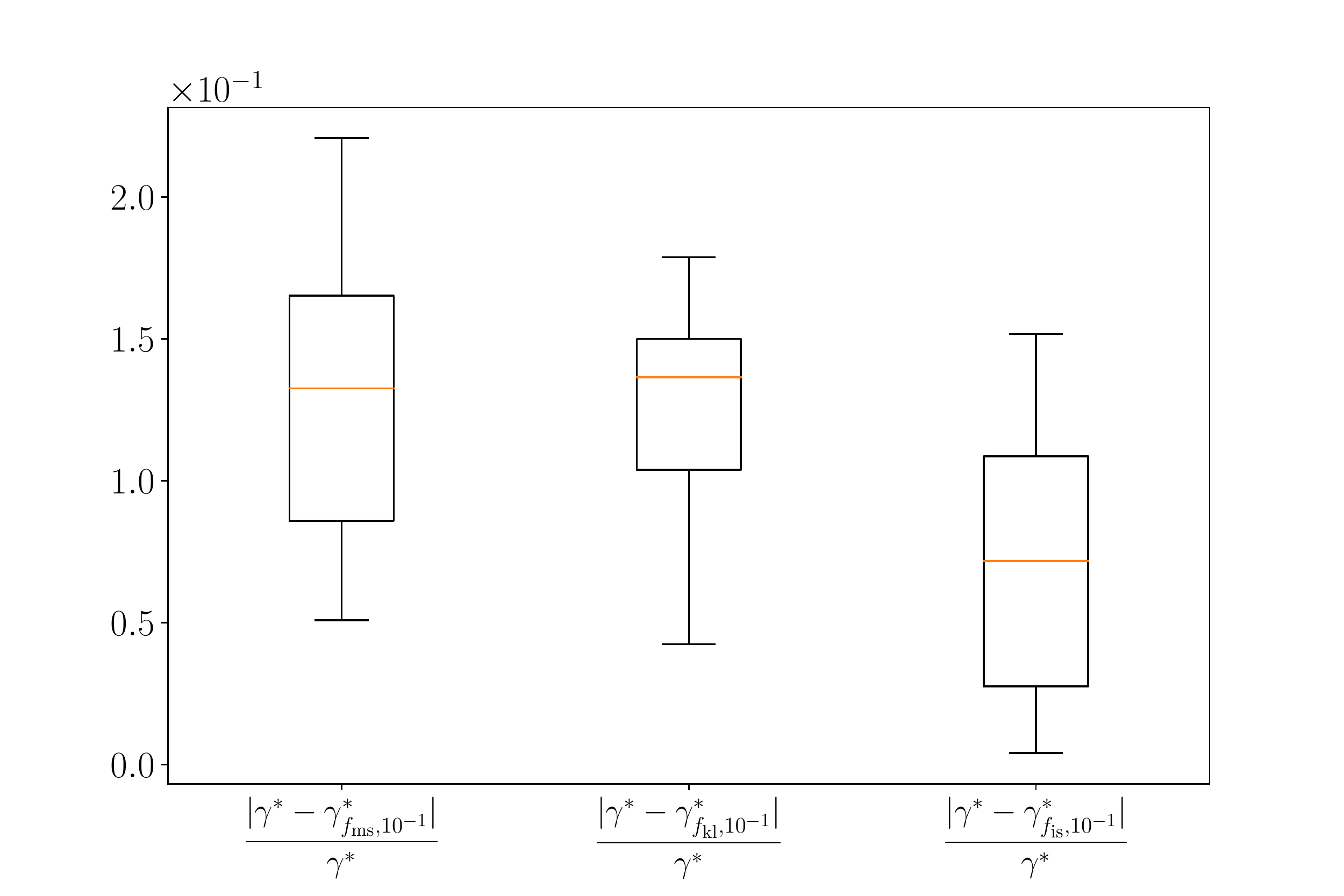}%
      \caption{Boxplots of the relative differences $| \gamma^* - \gamma^*_{f, \epsilon} | / \gamma^*$ between each minimizer $\gamma^*_{f, \epsilon}$ of P-UPBRE${}^f_\epsilon( \gamma )$ and the respective minimizer $\gamma^*$ of the squared error $\| \vect x^\gamma - \vect x^* \|_2^2$. Top left: mouse head with moderately noisy data simulated using the same discretization of the reconstruction (i.e, committing an inverse crime). Top right: mouse head with noisier data simulated using the same discretization than in the reconstruction. Bottom left: mouse head with moderately noisy data simulated using a discretization different from the one used in the reconstruction. Bottom right: Shepp-Logan phantom with more noise than in reconstructions shown in Figure~\ref{fig:box_minims}~and~\ref{fig:box_errors}, simulated using a discretization different from the one used in the reconstruction.}\label{fig:box_minims_outros}
   \end{figure}

   \section{Concentration of Measure}

   Thanks to the efforts of many, from the pioneering insights of Vitali Milman to the refined work of Michel Talagrand, the initial non-asymptotic concentration results of Chernoff and of Hoeffding evolved to the following general idea~\cite{tal96}:
   \begin{quote}
      ``A random variable that depends (in a `smooth' way) on the influence of many independent variables (but not too much on any of them) is essentially constant.''
   \end{quote}
   In order to understand the relevance of this phenomenon to the techniques described in the present paper, in what follows we will consider the consequences of concentrations of two different random variables that fit the description ``depends (in a `smooth' way) on the influence of many independent variables (but not too much on any of them)'' to our methodology.

   First we notice that there is in principle no reason to believe that $\expect_{\vect b} D_f\Bigl( \vect A( \vect x^* ), \vect A\bigl( B_\gamma( \vect b ) \bigr) \Bigr)$ would be close to $D_f\Bigl( \vect A( \vect x^* ), \vect A\bigl( B_\gamma( \vect b ) \bigr) \Bigr)$ for a given $\vect b$. As it is known, of course, the ultimate goal should be to minimize $D_f\Bigl( \vect A( \vect x^* ), \vect A\bigl( B_\gamma( \vect b ) \bigr) \Bigr)$, not $\expect_{\vect b} D_f\Bigl( \vect A( \vect x^* ), \vect A\bigl( B_\gamma( \vect b ) \bigr) \Bigr)$.    However, in many practical applications (such as imaging problems), $D_f\Bigl( \vect A( \vect x^* ), \vect A\bigl( B_\gamma( \vect b ) \bigr) \Bigr)$ is naturally smoothly dependent on several independent random variables (such as millions of data pixels $\vect b$) but not too much on any of them (each data pixel $b_i$ has limited influence in the final result). That is, the concentration of measure principle states that the probability that we have large
   \begin{equation}
      \Bigl| \expect_{\vect b} D_f\Bigl( \vect A( \vect x^* ), \vect A\bigl( B_\gamma( \vect b ) \bigr) \Bigr) - D_f\Bigl( \vect A( \vect x^* ), \vect A\bigl( B_\gamma( \vect b ) \bigr) \Bigr) \Bigr|
   \end{equation}
   is exponentially small.

   In practice, we do not compute $\expect_{\vect b} D_f\Bigl( \vect A( \vect x^* ), \vect A\bigl( B_\gamma( \vect b ) \bigr) \Bigr)$ directly. Instead, we make use of \eqref{eq:g-ureExact} or \eqref{eq:PupbreExact} because the quantities on the right-hand side of these equations are in principle computable without knowledge of the exact solution $\vect x^*$. Indeed, we do not even use $\expect D_f\bigl( \vect b, \vect A( \vect x^\gamma ) \bigr) + \sigma^2\expect\left[ \sum_{i = 1}^m \frac{\partial g_i}{\partial b_i}( \vect b ) \right]$ (we will use the Gaussian case for concreteness, but the discussion applies to the other cases too). Instead, we again reasonably assume that the probability is very small that
   \begin{equation}
      \left| \expect D_f\bigl( \vect b, \vect A( \vect x^\gamma ) \bigr) + \sigma^2\expect\left[ \sum_{i = 1}^m \frac{\partial g_i}{\partial b_i}( \vect b ) \right] - D_f\bigl( \vect b, \vect A( \vect x^\gamma ) \bigr) - \sigma^2\sum_{i = 1}^m \frac{\partial g_i}{\partial b_i}( \vect b ) \right|
   \end{equation}
   is large. Thus, taking \eqref{eq:g-ureExact} into consideration, we should expect that there is only a small probability that the following difference is large
   \begin{equation}
      \left| D_f\Bigl( \vect A( \vect x^* ), \vect A\bigl( B_\gamma( \vect b ) \bigr) \Bigr) - K - D_f\bigl( \vect b, \vect A( \vect x^\gamma ) \bigr) - \sigma^2\sum_{i = 1}^m \frac{\partial g_i}{\partial b_i}( \vect b ) \right|.
   \end{equation}

   \def\Exp{\expect}

   This is an admittedly vague affirmation. Indeed, we do not have the goal to provide concrete concentration inequalities here. Instead, our analysis will start from the following question: assuming that some function $\varphi : \mathbb R^m \times \mathbb R \to \mathbb R$ somehow concentrates around $\expect_{\vect b} \varphi( \vect b, \gamma )$, does the minimizer ${\overline\gamma}^*$ of $\varphi( \vect b, \gamma )$ concentrate around the minimizer $\gamma^*$ of $\expect_{\vect b} \varphi( \vect b, \gamma )$? In order to further simplify the analysis we will assume a discrete parameter space, leaving the continuous case for future research.

   Since we will be mainly concerned with $\varphi(\vect b, \gamma)$ as a function of $\gamma$, we denote $\varphi_{\vect b}(\gamma) := \varphi(\vect b, \gamma)$ from now on. We will assume that, for each $\vect b$, the function $\varphi_{\vect b}$ is locally Lipschitz continuous on the variable $\gamma$. Moreover, assume that there exists $\lambda > 0$ such that the function $\Exp\varphi_{\vect b}: \mathbb{R}_+\to \mathbb{R}$ satisfies:
   \begin{equation}
      \Exp\varphi_{\vect b}(\gamma) \geq \Exp\varphi_{\vect b}(\gamma^*) + \frac{\lambda}{2} \| \gamma - \gamma^* \|_2^2\text{, \quad for all }\gamma \geq 0\text{.}
   \end{equation}

   A sufficient condition for the above inequality to hold is strong convexity of $\Exp\varphi_{\vect b}( \gamma )$ with \elias{respect}{relation} to $\gamma$, but this is not \elias{necessary}{neccessary}, what is really required is that the function $\Phi( \gamma ) := \expect\varphi_{\vect b}( \gamma )$ is not too ``flat'' close to its (unique) minimizer $\gamma^*$. The following form of this inequality will be more frequently used below:
   \begin{equation}\label{eq:1}
      \Exp\varphi_{\vect b}(\gamma^*) - \Exp\varphi_{\vect b}(\gamma) \leq  - \frac{\lambda}{2} \| \gamma - \gamma^* \|_2^2\text{, \quad for all }\gamma \geq 0\text{.}
   \end{equation}

   We will assume that $\varphi_{\vect b}( \gamma )$ satisfies, uniformly over all possible values of $\gamma$, the following concentration inequalities for small enough $t > 0$:
   \begin{equation}\label{eq:concentration_upper}
      \prob[ \varphi_{\vect b}( \gamma ) > \expect \varphi_{\vect b}( \gamma ) + t ] \leq \exp\left( -\frac{t^2}{4V} \right)
   \end{equation}
   and
   \begin{equation}\label{eq:concentration_lower}
      \prob[ \varphi_{\vect b}( \gamma ) < \expect \varphi_{\vect b}( \gamma ) - t ] \leq \exp\left( -\frac{t^2}{4V} \right)
   \end{equation}
   for some $V > 0$. For conditions that might ensure the validity of such inequalities, see, e.g., \cite[Theorem~8.2]{dup09}.

   \def\Prob{\prob}

   Let us then compute a uniform bound on the probability that $\varphi_{\vect b}(\gamma)$ is larger than $\Exp\varphi_{\vect b}(\gamma^*)$ by  $c > 0$ units:
   \begin{equation}
      \begin{split}
         \Prob\left[ \varphi_{\vect b}(\gamma) \geq \Exp\varphi_{\vect b}(\gamma^*) + c \right] &{} =
         \Prob\left[ \varphi_{\vect b}(\gamma) - \varphi_{\vect b}(\gamma^*) + \varphi_{\vect b}(\gamma^*) \geq \Exp\varphi_{\vect b}(\gamma^*) + c \right] \\
         &{} = \Prob\left[ \varphi_{\vect b}(\gamma^*) \geq \Exp\varphi_{\vect b}(\gamma^*) + \varphi_{\vect b}(\gamma^*) - \varphi_{\vect b}(\gamma) + c \right] \\
         &{} \leq \Prob\left[ \varphi_{\vect b}(\gamma^*) \geq \Exp\varphi_{\vect b}(\gamma^*) - L\|\gamma - \gamma^*\| + c \right]\text,
      \end{split}
   \end{equation}
   where the last inequality comes from  the Lipschitz property of $\varphi_{\vect b}$. If $c > L\|\gamma - \gamma^*\|$, then we can use \eqref{eq:concentration_upper} and obtain
   \begin{equation}
      \Prob\left[ \varphi_{\vect b}(\gamma) \geq \Exp\varphi_{\vect b}(\gamma^*) + c \right] \leq \exp\left( - \frac{(c - L\|\gamma - \gamma^*\|)^2}{4V} \right)\text{.}
   \end{equation}

   Now, let us compute a bound for the probability that $\varphi_{\vect b}(\gamma)$ is smaller than $\Exp\varphi_{\vect b}(\gamma^*)$ by a difference of $0 < c < \lambda\|\gamma - \gamma^*\|_2^2/2$:

   \begin{equation}
      \begin{split}
         \Prob\left[ \varphi_{\vect b}(\gamma) \leq \Exp\varphi_{\vect b}(\gamma^*) + c \right] &{} =
         \Prob\left[ \varphi_{\vect b}(\gamma) \leq \Exp\varphi_{\vect b}(\gamma) + \Exp\varphi_{\vect b}(\gamma^*) - \Exp\varphi_{\vect b}(\gamma) + c \right] \\
         &{} \overset{\eqref{eq:1}}{\leq}  \Prob\left[ \varphi_{\vect b}(\gamma) \leq \Exp\varphi_{\vect b}(\gamma)  - \frac{\lambda}{2} \| \gamma - \gamma^* \|_2^2 + c \right] \\
         &{} \overset{\eqref{eq:concentration_lower}}{\leq} \exp\left( - \frac{(\frac{\lambda}{2} \| \gamma - \gamma^* \|_2^2 - c)^2}{4V} \right)\text.
      \end{split}
   \end{equation}

   So, suppose that we have a closed interval $[\gamma_0, \gamma_\ell]\subset \mathbb{R}_+$ such that $\gamma^* \in [\gamma_0, \gamma_\ell]$, and that $\gamma_0 < \gamma_1 < \hdots < \gamma_\ell$ produces a partition to $[\gamma_0, \gamma_\ell]$. Also, for some $d > 0$, consider the following two sets:
   \begin{equation}
      \Gamma := \{ \gamma \geq 0 : \| \gamma - \gamma^* \| \leq d \} \text{ and } \tilde\Gamma := \{ \gamma \geq 0 : \| \gamma - \gamma^* \| > d \}\text.
   \end{equation}
   Then, let us define the following events for any given $Ld < c < \lambda d^2/2$:
   \begin{itemize}
      \item{$A$: there is at least one  element in $\{\gamma_0, \gamma_1, \hdots, \gamma_\ell\}$ such that $\gamma_j \in \Gamma$;}
      \item{$B$: there is at least one element $\gamma_j \in \{\gamma_0, \gamma_1, \hdots, \gamma_\ell\}$ such that it belongs to $\Gamma$, and additionally, $\varphi_{\vect b}(\gamma_j) < \Exp\varphi_{\vect b}(\gamma^*) + c$; }
      \item{$C$: all elements $\gamma_i \in \{\gamma_0, \gamma_1, \hdots, \gamma_\ell\}$ that belong to $\tilde\Gamma$ satisfy $\varphi_{\vect b}(\gamma_i) > \Exp\varphi_{\vect b}(\gamma^*) + c$. }
   \end{itemize}
   Hence, letting
   \begin{equation}
      \hat\gamma\in \argmin_{\gamma\in \{\gamma_0, \hdots, \gamma_\ell\}} \varphi_{\vect b}(\gamma)\text,
   \end{equation}
   we have
   \begin{equation}
      \begin{split}
         \Prob\left[ \hat\gamma \in \Gamma | A \right] &{} \geq \Prob\left[ B \cap C | A \right] \\
         &{} = 1 - \Prob\left[ \neg B \cup \neg C | A \right] \\
         &{} \geq 1 - \Prob\left[ \neg B | A \right] - \Prob\left[ \neg C | A \right] \\
         &{} \geq 1 - \exp\left( - \frac{(c - Ld)^2}{4V} \right) - \exp\left( - \frac{(\frac{\lambda}{2} d^2 - c)^2}{4V} \right)\text.
      \end{split}
   \end{equation}
   In case that $d > 0$ is large enough for $Ld < (\lambda/4) d^2 < \lambda d^2/2$ to hold, then we can take $c = (\lambda/4) d^2$, and find a lower bound for $\Prob\left[ \hat\gamma \in \Gamma | A \right]$ in terms of $d > 0$:
   $$ \Prob\left[ \hat\gamma \in \Gamma | A \right] \geq 1 - \exp\left( - \frac{(\frac{\lambda}{4}d^2 - Ld)^2}{4V} \right) - \exp\left( - \frac{(\frac{\lambda}{4} d^2)^2}{4V} \right)\text{.}$$

\elias{
With the aim of providing an illustration for these ideas, consider the function 
$\varphi(\vect b, \gamma) : \mathbb{R}^2 \times \mathbb{R}\rightarrow \mathbb{R}$ stated as
$ \varphi(\vect b, \gamma) = (b_1 +1)^2 \gamma^2 - (b_2+1)^2 \gamma + \frac{b_1^2}{2}+ 2 b_2^2$, 
in which $ b_j \in \mathcal{N}(0, \sigma^2), \ j = 1, 2$.

By taking the standard deviations $\sigma \in \{0.1, 0.2, 0.3\}$ and randomly generating samples of 200 points in each case, 
the concentration of measure phenomenon may be observed in the plots of Figure~\ref{fig:graphs}. The sets of
  minimizers of $\varphi_{\vect b} (\gamma) := \varphi(\vect b, \gamma) $, given by  
$ \bar{\gamma}^* = \frac{(b_2+1)^2}{2(b_1 +1)^2}$, are displayed in Figure~\ref{fig:distrib} for 
each sampling. Notice that, since the average vector of coefficients is $\bar{\vect{b}} = (0,0)$, we obtain 
$\varphi_{\bar{\vect b}} (\gamma) = \gamma^2 - \gamma$, whose 
minimizer is 0.5. The minimizers of $\mathbb{E}\varphi_{\vect b}$ for each choice of the deviation~$\sigma$, namely 
 0.499592,   0.503212, and 0.610599, are closer to the ideal value of 0.5
 than the average of the minimizers of $\varphi_{\vect b} (\gamma) $, which are
0.519001,  0.65556, and 2.34117. We should stress that although the ideal value is 0.5, our interest here is to 
compare how close a minimizer of a single realization of $\varphi_{\vect b}$ is from the minimizer of 
$\mathbb{E}\varphi_{\vect b}$, since this is the subject of concern in this section.

For such an example, the constant of strong conexity may be set as $\lambda=2$. Denoting by~$r$ the 
radius of the neighborhood   to establish the local Lipschitz constant for $\varphi_{\vect b} (\gamma) $ 
around~$\gamma^*$, we reach $L=(1+2r)(1+\sigma)^2$. Setting $r=0.1$, the three choices for
 $\sigma$ yield $L \in \{1.452, 1.728, 2.028 \}$. Defining 
 $V=\sigma^2$ and $d = 9L/(2 \lambda)$, the desired relationships $Ld < (\lambda/4) d^2 < \lambda d^2/2$ hold. Thus,
the lower bounds we have computed for 
 $\mathbb{P}[\hat{\gamma} \in \Gamma | A]$ are respectively given by 0.999848, 0.987820, and 0.975685.

\begin{figure}[h]
\begin{tabular}{ccc}
\includegraphics[width=0.3\textwidth]{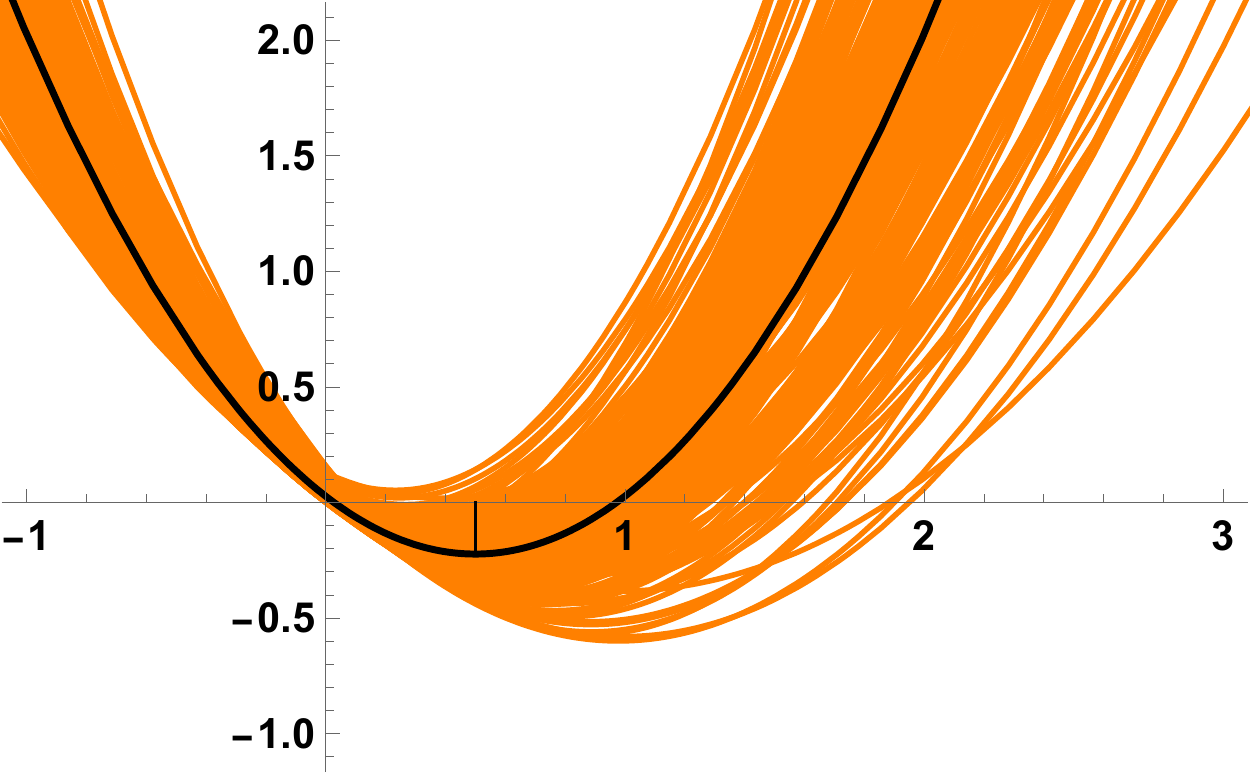}&
\includegraphics[width=0.3\textwidth]{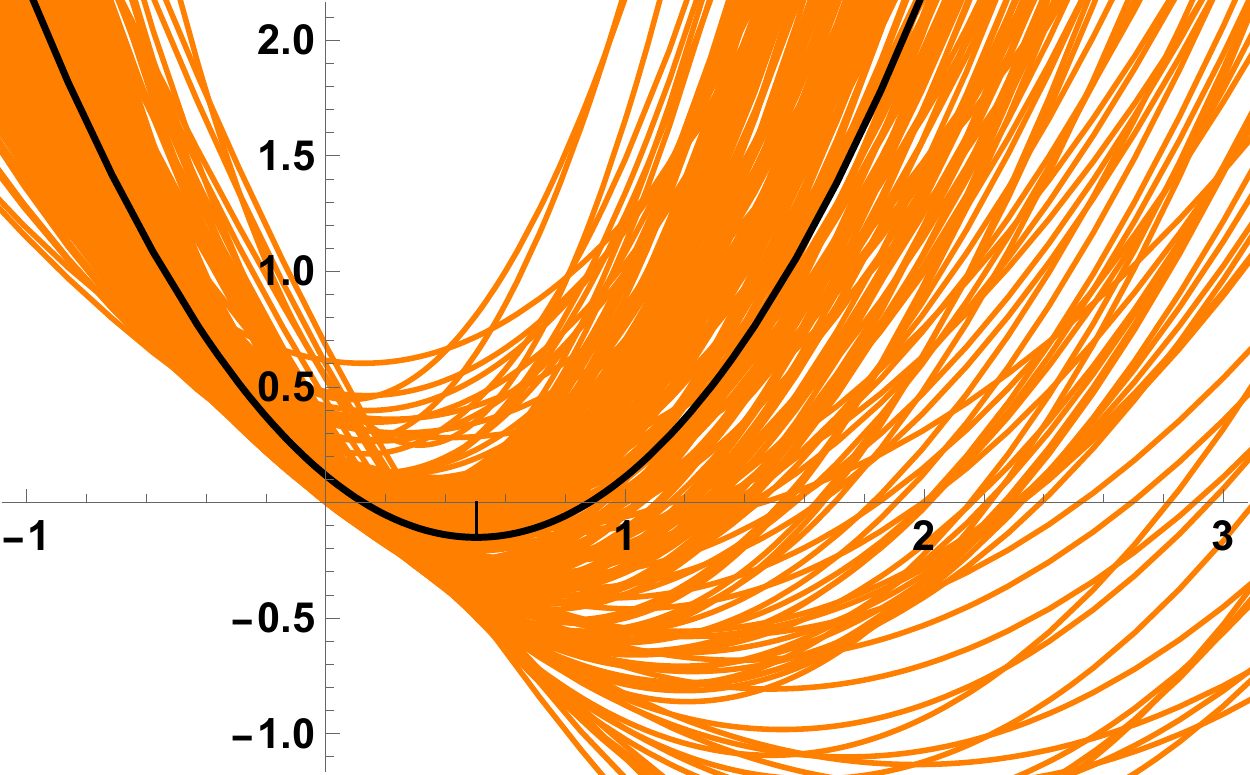}&
\includegraphics[width=0.3\textwidth]{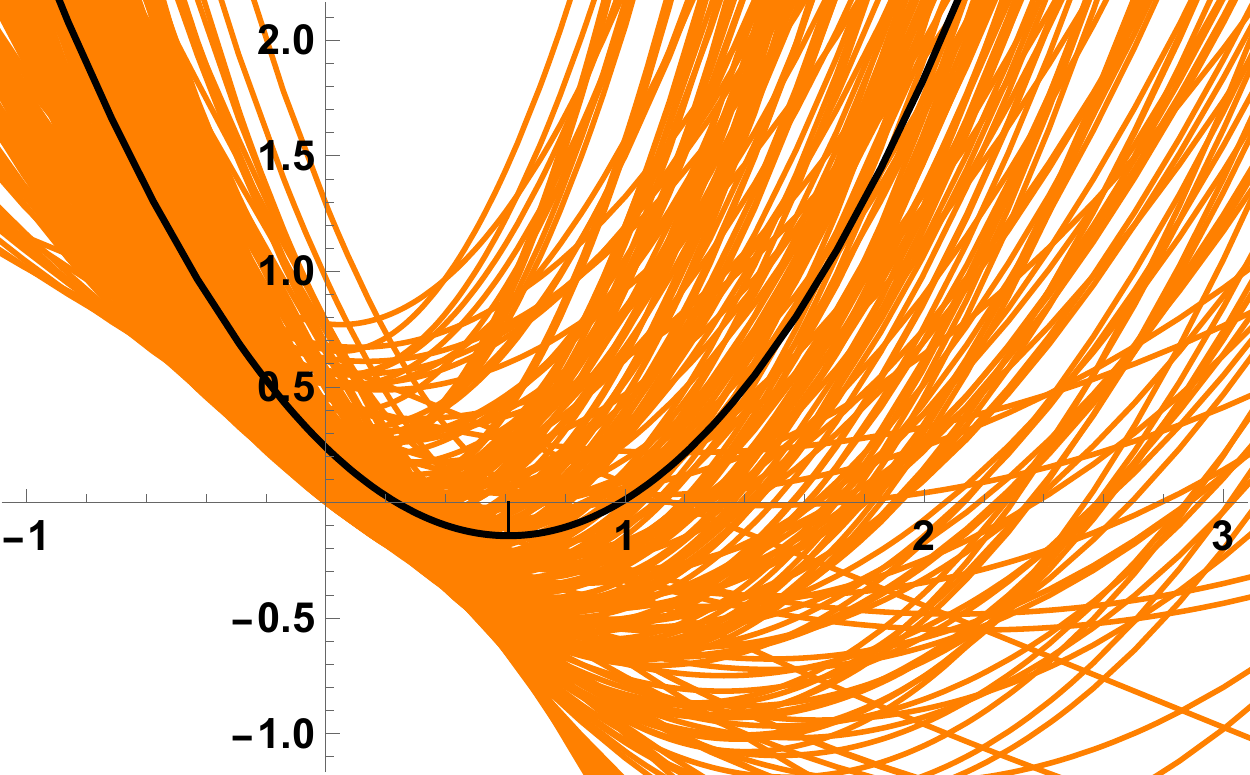}\\
$\sigma = 0.1$ &  $\sigma = 0.2$ & $\sigma = 0.3$
\end{tabular}
\caption{Graphs of the sampled functions $\varphi_{\vect b} (\gamma) $, together with the curve 
$y=\mathbb{E}\varphi_{\vect b}(\gamma)$.}
\label{fig:graphs}
\end{figure}

\begin{figure}[h]
\begin{tabular}{ccc}
\includegraphics[width=0.3\textwidth]{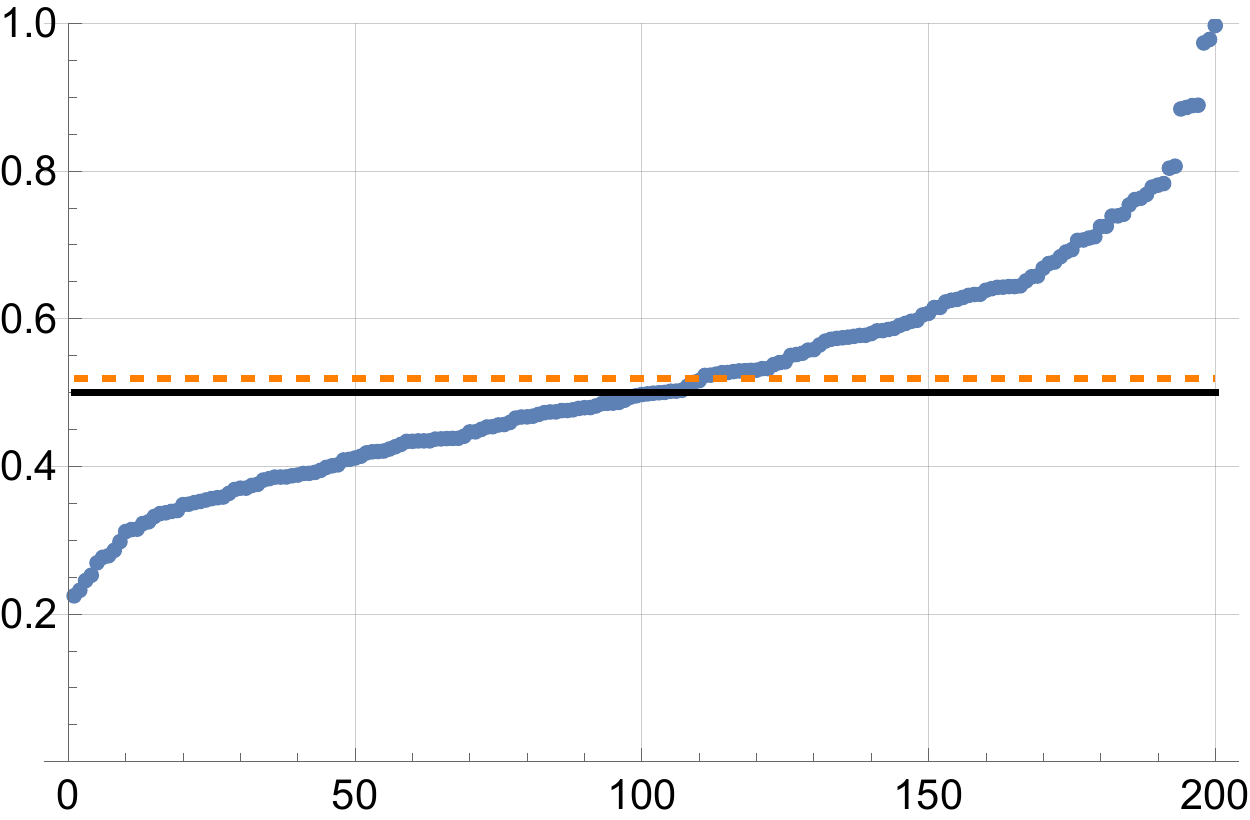}&
\includegraphics[width=0.3\textwidth]{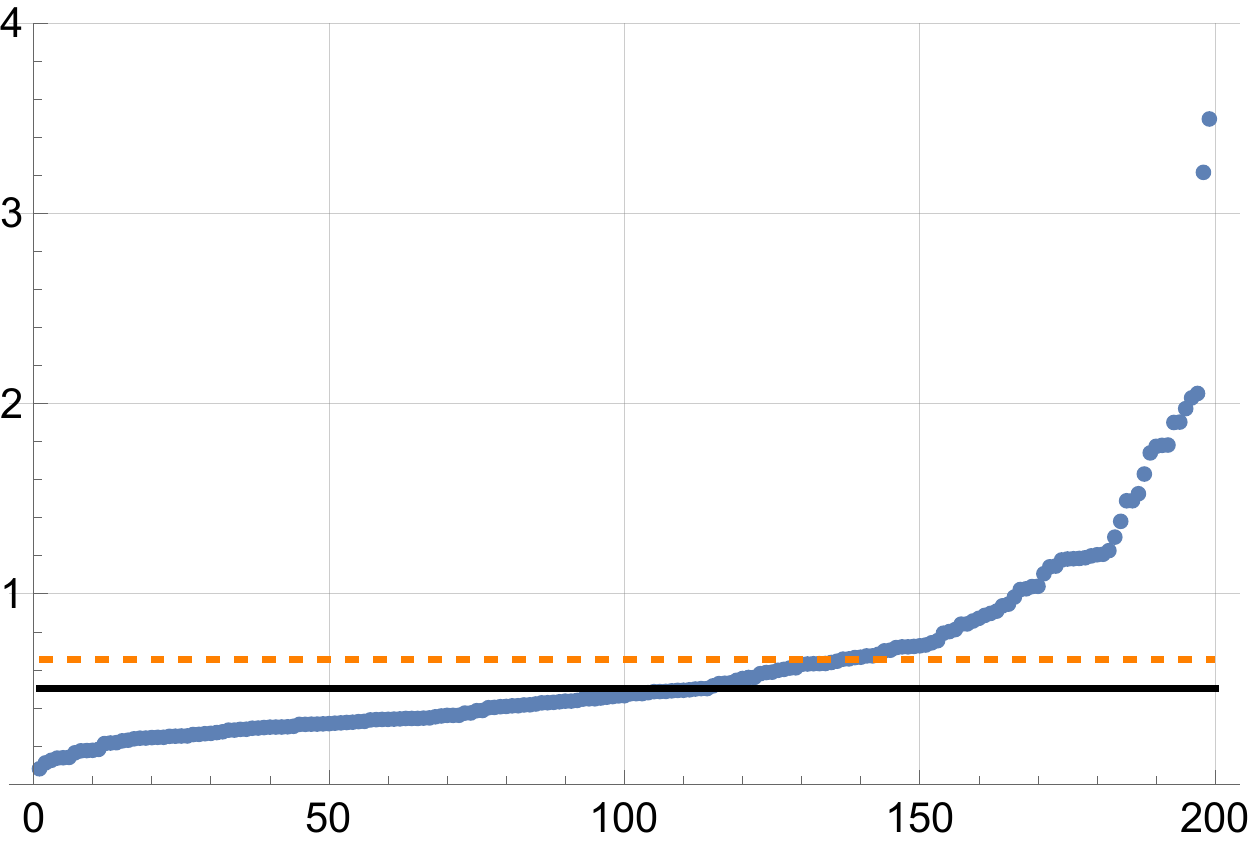}&
\includegraphics[width=0.3\textwidth]{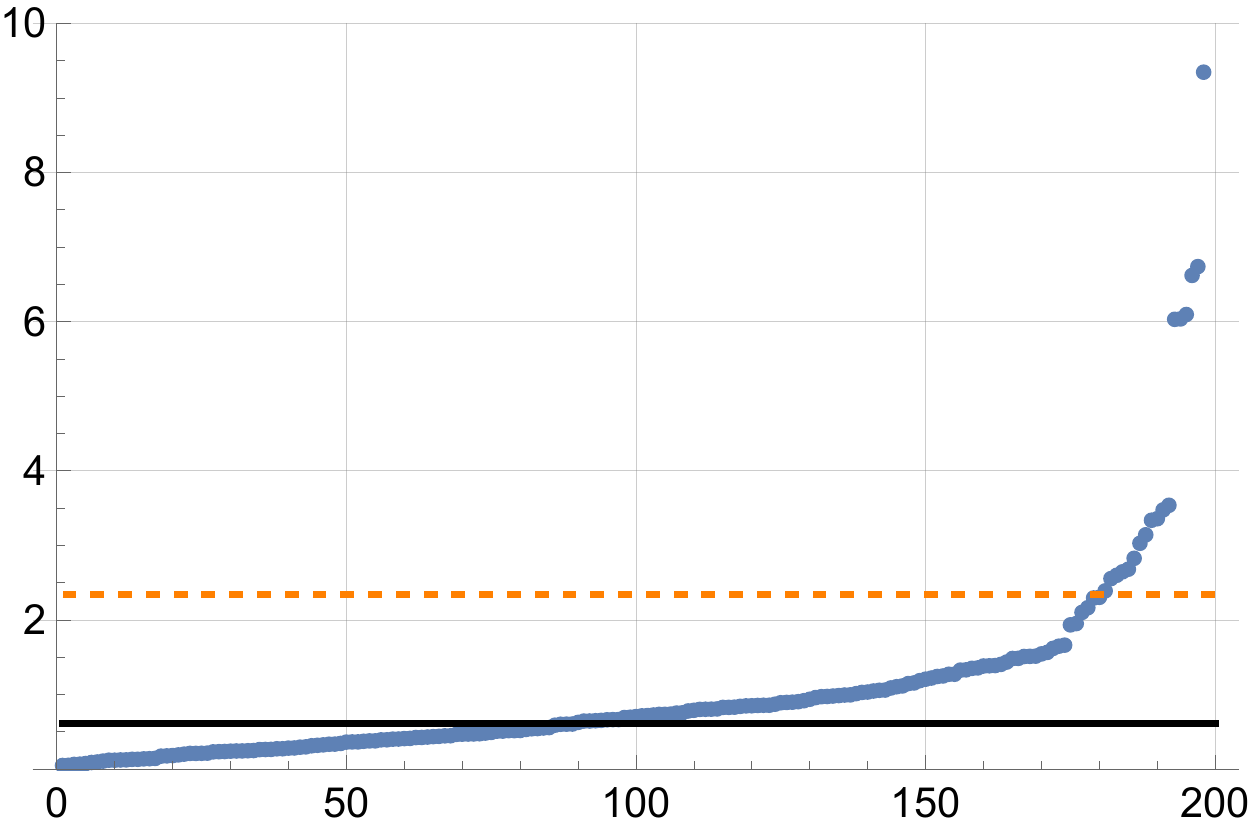}\\
$\sigma = 0.1$ &  $\sigma = 0.2$ & $\sigma = 0.3$
\end{tabular}
\caption{Sorted distribution of minimizers of $\varphi_{\vect b} (\gamma) $ with its average depicted in the dashed line. 
The minimizer of $\mathbb{E}\varphi_{\vect b}$ corresponds to the continuous horizontal line. Notice that the vertical ranges are distinct.}
\label{fig:distrib}
\end{figure}
}{}

   \section{Conclusions}

   We have introduced a generalization of the unbiased predictive risk estimator which allows the use of more general Bregman divergences than the squared norm of the difference. The minimization of these estimators \elias{leads}{lead} to a regularization parameter selection method for inverse problems, which we have applied to the image reconstruction problem in computed tomography. Simulated and real-world experiments corroborate the intuition that the flexibility to select the most appropriate Bregman divergence for the problem in hand might be useful. Finally, we have analysed what the consequences of a concentration inequality in the estimator would be for its minimizer and we concluded that if the estimator does indeed concentrate around its expected value, than so does the minimizer of this estimator.

   \section*{Acknowledgments}

   This work was partially funded by \emph{Conselho Nacional de Desenvolvimento Científico e Tecnológico} (CNPq) grants 310893/2019-4 and 305010/2020-4 and \emph{Fundação de Amparo à Pesquisa do Estado de São Paulo} (FAPESP) grants 2018/24293-0, 2016/22989-2 and 2013/07375-0.

   We are also indebted to Prof. Juliana Cobre for reading and commenting on an early version of the manuscript\elias{, and to the anonymous referees for pointing rooms for improving the presentation of our work.}{.}

   \bibliographystyle{paper}
   \bibliography{2020_ip.bib}

\end{document}